\documentclass[11pt,a4paper]{article}
\usepackage[utf8]{inputenc}
\usepackage{amsmath}
\usepackage{amsfonts}
\usepackage{amssymb}
\usepackage{amsthm}
\usepackage{enumerate}
\PassOptionsToPackage{normalem}{ulem}
\usepackage{ulem}
\usepackage[unicode=true,colorlinks,citecolor=linkcolor,linkcolor=linkcolor,urlcolor=urlcolor]{hyperref}
\usepackage[left=3cm,right=3cm,top=2.5cm,bottom=2.5cm]{geometry}
\usepackage{array}
\usepackage{cases}

\usepackage{natbib}
\usepackage{color}
\usepackage{dsfont}

\theoremstyle{plain}
\newtheorem{theorem}{\protect Theorem}[section]

\newtheorem{lemma}[theorem]{\protect Lemma}
\newtheorem{remark}[theorem]{\protect Remark}

\newtheorem{example}[theorem]{\protect Examples}

\newcommand{\R}{\mathbb{R}}
\newcommand{\Px}{\mathbb{P}}
\newcommand{\Qx}{\mathbb{Q}}
\newcommand{\Fx}{\mathbb{F}}
\newcommand{\Ex}{\mathbb{E}}
\newcommand{\F}{\mathcal{F}}
\newcommand{\N}{\mathbb{N}}
\newcommand{\I}{\mathds{1}}
\newcommand{\BRn}{{\cal B}_{D}}

\newcommand{\lc}{\langle}
\newcommand{\rc}{\rangle}

\newcommand{\Aa}{{\bf(A$_{b,\sigma}$)}}
\newcommand{\Ag}{{\bf(A$_{g}$)}}
\newcommand{\Ax}{{\bf(A$_{X_0}$)}}
\newcommand{\Agl}{{\bf(A$_{gl}$)}}
\newcommand{\Ab}{{\bf(A$_{X}$)}}

\newcommand{\AD}{{\bf(A$_{D,\sigma}$)}}
\newcommand{\Absig}{{\bf(A$_{b,\sigma}'$)}}
\newcommand{\Acir}{{\bf(A$_{\rm 1d}$)}}

\allowdisplaybreaks[4]

\definecolor{linkcolor}{rgb}{0,0,0.502}
\definecolor{urlcolor}{rgb}{1,0,0}

\begin{document}

\title{Probabilistic Analysis of Replicator-Mutator Equations}
\author{Lijun Bo\thanks{Email: lijunbo@ustc.edu.cn, School of Mathematical Sciences, University of Science and Technology of China, Hefei, Anhui
Province, 230026, China, and Wu Wen Tsun Key
Laboratory of Mathematics, Chinese Academy of Science, Hefei, Anhui Province 230026, China.}
\and
Huafu Liao\thanks{E-mail: lhflhf@mail.ustc.edu.cn, School of Mathematical Sciences, University of Science and Technology of China, Hefei, Anhui
Province, 230026, China.}}

\maketitle

\begin{abstract}
This paper introduces a general class of Replicator-Mutator equations on a multi-dimensional fitness space. We establish a novel probabilistic representation of weak solutions of the equation by using the theory of Fockker-Planck-Kolmogorov (FPK) equations and a martingale extraction approach. The examples with closed-form probabilistic solutions for different fitness functions considered in the existing literature are provided. We also construct a particle system and prove a general convergence result to any solution to the FPK equation associated with the extended Replicator-Mutator equation with respect to a Wasserstein-like metric adapted to our probabilistic framework.
\vspace{0.2 cm}

\vspace{0.2 cm}

\end{abstract}

\section{Introduction}

In this paper, we introduce a general class of Replicator-Mutator equations on a fitness space specified as a domain of $\R^n$ with $n\geq1$.
The classical replicator-mutator equation in evolutionary genetics described as an integro-differential Cauchy problem on the one-dimensional fitness space $\R$ is proposed by \cite{Kimura65}. This type of equations with the linear fitness function is used to model the evolution of RNA virus populations by \cite{Tsimringetal96} based on a mean-field approach. For the case where some phenotypes are infinitely well-adapted (this typically corresponds to the fitness functions which are unbounded from above), a so-called replicator-mutator model with cut-off at large phenotype is studied by \cite{RouzineWakekeyCoffin03} and \cite{SniegowskiGerrish10}, while \cite{RouzineBrunetWilke08} propose a proper stochastic treatment for large phenotypic trait region for the linear fitness case.

In recent years, rigorous mathematical treatments on existence and behaviours of solutions of Replicator-Mutator equations describing as a Cauchy problem on the one-dimensional fitness space for different types of fitness functions are developed. \cite{alfarocarles14} reduce the equation with the linear fitness function to a standard heat equation by applying a tricky transform of the solution based on the Avron-Herbst formula, and therefore, one can compute its solution explicitly. For the quadratic fitness case, a earlier work \cite{Fleming79} studies a model for inheritance of continuous polygenic traits and the equilibrium density of gametic types is found approximately. \cite{Burger98} provides an overview of the mathematical properties of various deterministic mutation-selection models. Recently, \cite{alfarocarles17} can also transform the equation to a heat equation by using a generalized lens transform of the Schr\"odinger equation. They also prove that, for any initial density, there is always extinction of the equation with the positive quadratic fitness function at a finite time. \cite{alfaroveruete19} establish an explicit solution representation of the equation as a non-local Cauchy problem via the underlying Schr\"odinger spectral elements when the fitness function is confining, and further a result on the long time behaviour related to that in \cite{Burger88} in terms of the principal eigenfunction is proved therein.

The methods used in the papers reviewed above rely on the solution-based transform or the eigenfunction expansion of operators. Our approach used in this paper to study the Replicator-Mutator equations is from a completely probabilistic perspective. To be more precise, the first goal of this paper is to provide a novel probabilistic representation of the (weak) solution of a general class of Replicator-Mutator equations on a multi-dimensional fitness space. We prove that the solution of the extended Replicator-Mutator equation can be in fact expressed in terms of probability transition density functions for some It\^o diffusion processes. The method used in this paper is jointly based on the probability representation of solutions to a class of Fockker-Planck-Kolmogorov (FPK) equations and a martingale extraction approach (see \cite{qinLinetsky16} where this approach is applied to the pricing of derivatives). Motivated by the probabilistic solution representation of the equation, we design a particle system in which the state process of homogeneous particles follows a system of stochastic differential equations (SDEs). The second goal of this paper is to establish a general convergence result (with respect to a Wasserstein-like distance adapted to our probabilistic framework) to any solution to our FPK equation associated with the extended Replicator-Mutator equation. In particular, easy byproducts of this result are uniqueness of solutions to the extended Replicator-Mutator equation and an explicit convergence rate to the solution of the equation. The strategy used in this paper is based upon a propagator method (see \cite{DelMoralMiclo2000}, \cite{Sznitman91} and \cite{Xu2018}). We also provide examples which admit the closed-form probabilistic solution of the extended Replicator-Mutator equation for different fitness functions considered in the aforementioned works.

The paper is organized as follows: we introduce in Section~\ref{sec:RMmodel} the extended Replicator-Mutator equation under a probabilistic framework. Section~\ref{sec:FPKequationexample} establishes a novel probabilistic representation of the (weak) solution of the extended Replicator-Mutator equation and provide examples with different fitness functions considered in the existing literature. Section~\ref{sec:propagation} constructs a particle system and prove a general convergence result to any solution to the FPK equation associated with the extended Replicator-Mutator equation. Thus result is usually called {\it propagation of chaos} which was first formulated by \cite{Kac1956}.

\section{Extended Replicator-Mutator equations}\label{sec:RMmodel}

In this section, we introduce a class of extended Replicator-Mutator (RM) equations as a non-local Cauchy problem on a domain in $\R^n$ for $n\geq1$ using a probabilistic framework.

\subsection{It\^o diffusion process}
This section presents a class of SDEs which is related to the representation of our extended RM equation on a domain $D$ in $\R^n$. Let $T>0$ be an arbitrary fixed time horizon and $D$ a domain in $\R^n$, i.e., an open connected subset of $\R^n$. Consider the continuous functions $b:D\to\R^{n\times1}$ and $\sigma:D\to\R^{n\times m}$, where $m\geq1$. We then introduce the following second-order differential operator acted on $C^2(D)$, which is given by, for $f\in C^2(D)$,
\begin{align}\label{eq:operatorsLn}
{\cal A}f(x) &:= b(x)^{\top}\nabla_xf(x) + \frac{1}{2}{\rm tr}[\sigma\sigma^{\top}(x)\nabla_x^2f(x)],\quad x\in D,
\end{align}
where $\nabla_x=(\partial_{x_1},\ldots,\partial_{x_n})^{\top}$ denotes the gradient operator, $\nabla_x^2$ is the corresponding Hessian matrix and ${\rm tr}$ denotes the trace operator. The operator ${\cal A}$ acted on $C^2(D)$ is in fact the infinitesimal generator of the (Markovian) solution of SDE given by, for $(t,x)\in[0,T]\times D$,
\begin{align}\label{eq:SDEXtx}
dX_s^{t,x} &= b(X_s^{t,x})ds + \sigma(X_s^{t,x})dW_s,~s\in[t,T]\nonumber\\
X_t^{t,x}&=x\in D,
\end{align}
where $W=(W_t)_{t\in[0,T]}$ is an $m$-dimensional Brownian motion on the filtered probability space $(\Omega,\F,\Fx,\Px)$ with the filtration $\Fx=(\F_t)_{t\in[0,T]}$ satisfying the usual conditions.

We impose the following assumptions on the coefficients of SDE \eqref{eq:SDEXtx} so to guarantee the existence a unique strong solution of the equation:
\begin{description}
  \item[{\Aa}](i) $b:D\to\R^{n\times1}$ and $\sigma:D\to\R^{n\times m}$ are locally Lipschitiz continuous; or (ii) for $m=n=1$, $b:D\to\R$ is locally Lipschitiz continuous and $\sigma:D\to\R$ is H\"{o}lder continuous with exponent $\gamma\in[\frac{1}{2},1)$.
  \item[{\Ab}] For all $(t,x)\in[0,T]\times D$, the solution $X^{t,x}=(X_s^{t,x})_{s\in[t,T]}$ of \eqref{eq:SDEXtx} neither explodes nor leaves $D$ before $T$, i.e., $\Px(\sup_{s\in[0,T]}|X_s^{t,x}|)=1$ and $\Px(X_s^{t,x}\in D,~\forall~s\in[t,T])=1$.
\end{description}

By Theorem 5.2.5 of \cite{karatzasshreve1991}, the assumption {\Aa}-(i) implies that SDE~\eqref{eq:SDEXtx} admits a unique strong solution $X^{t,x}$ up to a possibly finite random explosion time. Then, the assumption {\Ab} yields that this explosion time should be greater than $T$, $\Px$-a.s., and hence $X^{t,x}$ is well-defined on $[t,T]$. For the case with $m=n=1$ (i.e., SDE~\eqref{eq:SDEXtx} is a one-dimensional equation), by Proposition 5.2.13 of \cite{karatzasshreve1991}, the assumption {\Aa}-(ii) gives that SDE~\eqref{eq:SDEXtx} admits a unique strong solution $X^{t,x}$ up to a possibly finite random explosion time. Therefore, the assumption {\Ab} yields that this explosion time is greater than $T$, a.s., and hence $X^{t,x}$ is well-defined on $[t,T]$.

\subsection{Extended RM equations on $[0,T]\times D$}

This section introduces a class of extended RM equations as a non-local Cauchy problem on $[0,T]\times D$. More precisely, let ${\cal A}^*$ be the adjoint operator of ${\cal A}$ defined by \eqref{eq:operatorsLn}\footnote{It can be seen that if $(b,\sigma)$ are sufficiently smooth, then ${\cal A}^*f(x)=\frac{1}{2}\sum_{i,j=1}^n\partial_{x_ix_j}^2((\sigma\sigma^{\top})_{ij}(x)f(x))-\sum_{i=1}^n\partial_{x_i}(b_i(x)f(x))$ for $x\in D$.}. Then, the extended RM equation considered in this paper, is given by, for $(t,x)\in(0,T]\times D$,
\begin{align}\label{eq:RM-eqn-Rn}
\left\{
  \begin{array}{ll}
    \displaystyle \partial_tu(t,x)={\cal A}^*u(t,x)+\left(g(x)-\int_{D}g(y)u(t,y)dy\right)u(t,x);\\ \\
    \displaystyle u(0,x)=u_0(x),
  \end{array}
\right.
\end{align}
where $u_0:D\to\R_+$ is a probability density function, i.e., $\int_D u_0(x)dx=1$, and $g:D\to\R$ is referred to as the {\it fitness function} (see, e.g. \cite{Kimura65} and \cite{Tsimringetal96}). The condition satisfied by the fitness function $g$ will be imposed later.

We next give the definition of a weak solution of the above extended RM equation~\eqref{eq:RM-eqn-Rn}. To this purpose, let $C_0^{\infty}(D)$ be the space of infinitely differentiable functions with compact support. Then, a function $u:[0,T]\times D\to\R$ is called a {\it weak solution} of the RM equation \eqref{eq:RM-eqn-Rn} if it satisfies the following variational form: for all test functions $f\in C_0^{\infty}(D)$,
\begin{align}\label{eq:weak-EM-eqn}
\lc u(t),f\rc &=\lc u_0,f\rc +\int_0^t\lc u(s),({\cal A}+g)f\rc ds-\int_0^t\lc u(s),g\rc \lc u(s),f\rc ds,
\end{align}
where the integral $\lc u(t),f\rc:=\int_D u(t,x)f(x)dx$. One of objectives of the paper is to establish a closed-form probabilistic representation of the weak solution to the RM equation~\eqref{eq:RM-eqn-Rn} in the above distributional sense.  We provide sufficient conditions satisfied by the fitness function $g$ under which the solution of \eqref{eq:weak-EM-eqn} admits an explicit form which can be expressed in terms of the transition density function of some It\^o diffusion processes. The main strategy for achieving this aim is to provide a probabilistic representation of the solution of \eqref{eq:RM-eqn-Rn} by using the theory of FPK equations and a martingale extraction approach.

The second objective of this paper is to establish the {\it propagation of chaos} of our RM equation \eqref{eq:RM-eqn-Rn} for a relatively large class of fitness functions $g$. It is in fact equivalent to the convergence of the empirical measure of a (homogeneous) particle system to an arbitrary solution to the FPK equation related to \eqref{eq:weak-EM-eqn} with respect to a Wasserstein-like distance adapted to our probabilistic framework. We also establish a rate of convergence in the propagation of chaos. This in particular implies uniqueness in the class of all weak solutions of our RM equation~\eqref{eq:weak-EM-eqn}.

\section{FPK equation and examples}\label{sec:FPKequationexample}

This section introduces a class of FPK equations associated with the RM equation \eqref{eq:weak-EM-eqn}. We establish a probabilistic representation of the solution of the FPK equation and then the probabilistic solution of \eqref{eq:weak-EM-eqn} follows by assuming the absolute continuity of the initial date of the FPK equation w.r.t. Lebesgue measure. Finally, we provide examples with closed-form probabilistic solutions for different fitness functions considered in the existing literature.

\subsection{Probabilistic solution of FPK equation}
Denote by ${\cal P}(D)$ the set of probability measures on ${\cal B}_{D}$ (i.e., the $\sigma$-algebra generated by the open subsets of $D$). For $p\geq1$, let ${\cal P}_p(D)\subset{\cal P}(D)$ be the set of probability measures on ${\cal B}_{D}$ with finite $p$-order moment. We introduce the FPK equation associated with \eqref{eq:weak-EM-eqn}, which is given by, for $f\in C_0^{\infty}(D)$,
\begin{align}\label{eq:FPKeqn}
\lc\mu_t,f\rc = \lc\rho_0,f\rc+\int_0^t \lc\mu_s,({\cal A}+g)f\rc ds - \int_0^t\lc\mu_s,f\rc\lc\mu_s,g\rc ds,\quad t\in[0,T],
\end{align}
where the initial datum $\rho_0\in{\cal P}(D)$ and the integral $\lc\mu_t,f\rc:=\int_D fd\mu_t$ for $t\in[0,T]$.
\begin{remark}\label{rem:densityofFPK}
If the initial datum $\rho_0$ of the FPK equation \eqref{eq:FPKeqn} admits a density function given by $u_0(x)$ for $x\in D$, namely, $\rho_0(dx)=u_0(x)dx$, then, for any $t\in[0,T]$, the solution $\mu_t(dx)=u(t,x)dx$, where $u(t,x)$ for $(t,x)\in[0,T]\times D$ is the (weak) solution of the RM equation~\eqref{eq:weak-EM-eqn}.
\end{remark}

The following lemma provides a probabilistic representation of the ${\cal P}(D)$-valued solution $\mu=(\mu_t)_{t\in[0,T]}$ for the FPK equation.
\begin{lemma}\label{lem:solutionFPK}
Let assumptions {\Aa} and {\Ab} hold. Assume that the fitness function $g:D\to\R$ satisfies that
\begin{description}
  \item[{\Ag}] For any $(t,x)\in[0,T]\times D$, $\Ex[\exp(\int_0^tg(X_s^{x})ds)]<+\infty$. Moreover, for $t\in[0,T]$, $x\to\Ex[\exp(\int_0^tg(X_s^{x})ds)]$ is in $L^1(D;\rho_0)$, where $X_t^{x}:=X_{t}^{0,x}$ for $t\in[0,T]$.
\end{description}
Let us define that, for $t\in[0,T]$,
\begin{align}\label{eq:solutionmut}
\mu_t= \frac{\int_{D}\Ex\left[\delta_{X_t^y}\exp\left(\int_0^tg(X_s^y)ds\right)\right]\rho_0(dy)}{\int_{D}\Ex\left[\exp\left(\int_0^tg(X_s^y)ds\right)\right]\rho_0(dy)},
\quad {\rm on}~{\cal B}_{D},
\end{align}
where $\delta$ denotes Dirac-delta measure. Then, $\mu=(\mu_t)_{t\in[0,T]}$ defined by \eqref{eq:solutionmut} is a ${\cal P}(D)$-valued solution of the FPK equation \eqref{eq:FPKeqn}.
\end{lemma}

\begin{proof}
It is obvious to see that $\mu_t$ defined by \eqref{eq:solutionmut} is a probability measure on ${\cal B}_D$ for any $t\in[0,T]$, where $\mu_0=\rho_0$.
We define that
\begin{align*}
h_t:=\int_D\Ex\left[\exp\left(\int_0^tg(X_s^y)ds\right)\right]\rho_0(dy),\quad t\in[0,T],
\end{align*}
which is well-defined due to the assumption {\Ag}. By virtue of the representation \eqref{eq:solutionmut}, we have that, for $s\in[0,T]$,
\begin{align*}
h_s\lc\mu_s,g\rc=\int_{D}\Ex\left[g(X_s^y)\exp\left(\int_0^sg(X_r^y)dr\right)\right]\rho_0(dy).
\end{align*}
Integrate on both sides of the above equality w.r.t. $s$ from $0$ to $t\in(0,T]$, it follows from the Fubini's theorem that
\begin{align*}
\int_0^t h_s\lc\mu_s,g\rc ds&=\int_{D}\Ex\left[\int_0^tg(X_s^y)\exp\left(\int_0^sg(X_r^y)dr\right)ds\right]\rho_0(dy)\nonumber\\
&=\int_D\Ex\left[\exp\left(\int_0^tg(X_s^y)ds\right)\right]\rho_0(dy)-1=h_t-1.
\end{align*}
This yields the ODE given by $h_t'=h_t\lc\mu_t,g\rc$ with $h_0=1$. Then, it holds that, for $t\in[0,T]$,
\begin{align}\label{eq:explcmutgrc}
\exp\left(\int_0^t\lc\mu_s,g\rc ds\right)=\int_{D}\Ex\left[\exp\left(\int_0^t g(X_s^{y})ds\right)\right]\rho_0(dy).
\end{align}
It follows from \eqref{eq:solutionmut} and \eqref{eq:explcmutgrc} that
\begin{align}\label{eq:limitmu}
\mu_t &= \int_{\R}\Ex\left[\delta_{X_t^{y}}\exp\left(\int_0^t(g(X_s^{y})-\lc\mu_s,g\rc) ds\right)\right]\rho_0(dy),\quad {\rm on}~{\cal B}_{D}.
\end{align}
Then, for any test function $f\in C_0^{\infty}(D)$, by applying It\^o formula to $f(X_t^y)$, which yields that
\begin{align*}
&\Ex\left[f(X_t^y)\exp\left(\int_0^t(g(X_s^{y})-\lc\mu_s,g\rc) ds\right)\right]= f(y)\nonumber\\
&\qquad+\int_0^t \Ex\left[\exp\left(\int_0^s(g(X_r^{y})-\lc\mu_r,g\rc) dr\right){\cal A}f(X_s^y)\right]ds\nonumber\\
&\qquad+\int_0^t \Ex\left[\exp\left(\int_0^s(g(X_r^{y})-\lc\mu_r,g\rc) dr\right)f(X_s^y)(g(X_s^{y})-\lc\mu_s,g\rc)\right]ds.
\end{align*}
Integrate on both sides of the above equality w.r.t. $\rho_0(dy)$. Then, by virtue of \eqref{eq:limitmu}, we arrive at
\begin{align*}
\lc\mu_t,f\rc&=\lc\rho_0,f\rc+\int_0^t\lc\mu_s,{\cal A}f\rc ds+\int_0^t \lc\mu_s,gf-\lc\mu_s,g\rc f\rc ds.
\end{align*}
This yields the FPK equation~\eqref{eq:FPKeqn}.
\end{proof}

\subsection{Closed-form representation of $\mu$ given by \eqref{eq:solutionmut}}\label{eq:solFPKexam}

We study a closed-form representation of the ${\cal P}(D)$-valued function $\mu$ given by \eqref{eq:solutionmut} in Lemma~\ref{lem:solutionFPK} by applying the martingale approach. We find that the obtained explicit form of $\mu$ can be in fact represented in terms of the transition density function of a class of It\^o diffusion processes. Moreover, we provide examples in which the expressions of $\mu$ and its density admit complete closed-form.

By \eqref{eq:solutionmut}, the representation of $\mu$ can be reduced to identify the following function: for $f\in C_b(D)$, i.e., the space of bounded continuous functions on $D$, define
\begin{align}\label{eq:Phitx}
\Phi^{f}(t,x):= \Ex\left[f(X_t^{x})\exp\left(\int_0^t g(X_s^{x})ds\right)\right],\quad (t,x)\in[0,T]\times D,
\end{align}
where the process $X^{x}=(X_t^{0,x})_{t\in[0,T]}$ is the unique strong solution of SDE~\eqref{eq:SDEXtx}. We next apply a martingale approach to give a closed-form representation of $\mu$ when the coefficients of equation $(b,\sigma,g)$ satisfy an additional constraint.
\begin{lemma}\label{lem:simpcaseRMM}
Let assumptions {\Aa}, {\Ab} and {\Ag} hold. Assume additionally that
\begin{description}
  \item[{\rm(i)}] the fitness function $g\in C^2(D)$ and satisfies that, for all $x\in D$,
  \begin{align}\label{eq:Agcond}
  {\cal A}g(x)=C_1\in\R,~\text{and}~~\nabla_xg(x)^{\top}\sigma(x)=C_2\in\R^{1\times m},
  \end{align}
  where the operator ${\cal A}$ acted on $C^2(D)$ is given by \eqref{eq:operatorsLn}.
  \item[{\rm(ii)}] for $y\in D$, let $\overline{X}^y=(\overline{X}_t^y)_{t\in[0,T]}$ be the strong solution of the following SDE:
  \begin{align}\label{eq:SDEbarX0}
    \overline{X}_t^y = y+ \int_0^t\overline{b}(s,\overline{X}_s^y)ds + \int_0^t\sigma(\overline{X}_s^y)dW_s,
\end{align}
where the time-dependent function $\overline{b}(t,x):=b(x)-t\sigma(x)C_2^{\top}$ for $(t,x)\in[0,T]\times D$.
\end{description}
Then, for any $T>0$, the solution $\mu$ given by \eqref{eq:solutionmut} of the FPK equation \eqref{eq:FPKeqn} admits the following form:
\begin{align}\label{eq:mudx02}
\mu_t(dx) =\frac{e^{tg(x)}\int_{D}\Px(\overline{X}_t^y\in dx)\rho_0(dy)}{\int_{\R}e^{tg(z)}\int_{D}\Px(\overline{X}_t^y\in dz)\rho_0(dy)},\quad t\in[0,T].
\end{align}
Moreover, if $\rho_0(dx)=u_0(x)dx$ and the density function $\overline{p}(t,y;x)$ of $\overline{X}_t^y$ for $t\in[0,T]$ exists, namely
\[
\Px(\overline{X}_t^y\in dx)=\overline{p}(t,y;x)dx,
\]
then, for all $t\in[0,T]$, the probability measure $\mu_t(dx)=u(t,x)dx$. Here $u(t,x)$ satisfies the extended RM equation~\eqref{eq:weak-EM-eqn}, and it has the following representation given by
\begin{align}\label{eq:densityutx0}
u(t,x)=\frac{e^{tg(x)}\int_{D}\overline{p}(t,y;x)u_0(y)dy}{\int_{D}e^{tg(z)}\left(\int_{D}\overline{p}(t,y;z)u_0(y)dy\right)dz},\quad (t,x)\in[0,T]\times D.
\end{align}
\end{lemma}

\begin{proof}
Integration by parts yields that, for $(t,y)\in[0,T]\times D$,
\begin{align*}
\int_0^t g(X_s^y)ds &= tg(X_t^y)-\int_0^t s{\cal A}g(X_s^y)ds-\int_0^t s \nabla_xg(X_s^y)^{\top}\sigma(X_s^y)dW_s\nonumber\\
&=tg(X_t^y)+\int_0^t\{\frac{1}{2}s^2\nabla_xg(X_s^y)^{\top}(\sigma\sigma^{\top})(X_s^y)\nabla_xg(X_s^y)-s{\cal A}g(X_s^y)\}ds\nonumber\\
&\quad-\int_0^t s \nabla_xg(X_s^y)^{\top}\sigma(X_s^y)dW_s-\frac{1}{2}\int_0^ts^2\nabla_xg(X_s^y)^{\top}(\sigma\sigma^{\top})(X_s^y)\nabla_xg(X_s^y)ds.
\end{align*}
Note that, by \eqref{eq:Agcond} in the condition (i), we have that ${\cal A}g(x)\equiv C_1\in\R$ and $\nabla_xg(x)^{\top}\sigma(x)=C_2\in\R^{1\times m}$. Then, by verifying the Novikov's condition, we can define a $(\Px,\Fx)$-martingale as follows:
\begin{align*}
N_t:=\exp\left(-\int_0^t s C_2dW_s-\frac{C_2C_2^{\top}}{6}t^3\right),\quad t\in[0,T].
\end{align*}
Therefore, by virtue of \eqref{eq:Phitx}, for $(t,y)\in[0,T]\times D$,
\begin{align}\label{eq:PhiEN}
\Phi^f(t,y)=\Ex\left[N_tf(X_t^{y})\exp\left(tg(X_t^y)+\frac{C_2C_2^{\top}}{6}t^3-\frac{C_1}{2}t^2\right)\right].
\end{align}
Let us define $\frac{d\Qx}{d\Px}|_{\F_t}=N_t$, for $t\in[0,T]$. Then $B_t:=W_t+C_2^{\top}t$ for $t\in[0,T]$ is an $m$-dimensional Brownian motion under the probability measure $\Qx$. Moreover, under the new probability measure $\Qx$, the dynamics of the process $X^y=(X_t^y)_{t\in[0,T]}$ is given by, for $y\in D$,
\begin{align*}
X_t^y = y+ \int_0^t(b(X_s^y)-s\sigma(X_s^y)C_2^{\top})ds + \int_0^t\sigma(X_s^y)dB_s.
\end{align*}
By applying \eqref{eq:PhiEN}, we arrive at
\begin{align*}
\Phi^f(t,y)&=\exp\left(\frac{C_2C_2^{\top}}{6}t^3-\frac{C_1}{2}t^2\right)\Ex^{\Qx}\left[f(X_t^{y})e^{tg(X_t^y)}\right]\nonumber\\
&=\exp\left(\frac{C_2C_2^{\top}}{6}t^3-\frac{C_1}{2}t^2\right)\int_{D}f(x)e^{tg(x)}\Qx(X_t^y\in dx).
\end{align*}
Then, the solution representation \eqref{eq:solutionmut} yields that, for $t\in[0,T]$,
\begin{align}\label{eq:solutionmut11}
\int_{D} f(x)\mu_t(dx) = \frac{\int_{D}f(x)e^{tg(x)}\int_{D}\Qx(X_t^y\in dx)\rho_0(dy)}{\int_{D}e^{tg(x)}\int_{D}\Qx(X_t^y\in dx)\rho_0(dy)}.
\end{align}
Note that $\Qx(X_t^y\in dx)=\Px(\overline{X}_t^y\in dx)$. Therefore, the desired results \eqref{eq:mudx02} and \eqref{eq:densityutx0} follows from the equality \eqref{eq:solutionmut11}.
\end{proof}

Lemma~\ref{lem:simpcaseRMM} works well for the case where the drift and volatility functions $(b,\sigma)$ are constant matrix and the fitness function $g$ is a linear mapping. This is documented in the following example:
\begin{example}\label{eq:constdrift-volatility}
Consider $m=n$, $b(x)\equiv b\in\R^{n}$ and $\sigma(x)\equiv\sigma\in\R^{n\times n}$. Let $\sigma\in\R^{n\times n}$ be invertible and its invertible matrix is given by $\sigma^{-1}$. In this example, SDE~\eqref{eq:SDEXtx} is reduced to a drift-Brownian motion described as:
\begin{align*}
X_t^{x}=x+bt+\sigma W_t,\quad t\in[0,T].
\end{align*}
Therefore, the domain $D=\R^n$. For any $C_2\in\R^{1\times n}$, we consider the following fitness function given by
\begin{align}\label{eq:linear-fitness}
g(x) = C_2\sigma^{-1}x,\quad x\in\R^{n\times1}.
\end{align}
Then, a direct calculation yields that, for all $t\in[0,T]$,
\begin{align*}
  \Ex\left[\exp\left(\int_0^tg(X^x_s)ds\right)\right]&=e^{tC_2\sigma^{-1}x+\frac{t^2}{2}C_2\sigma^{-1}b}
  \Ex\left[\exp\left(tC_2W_t-C_2\int_0^tsdW_s\right)\right]\notag\\
  &=\exp\left(tC_2\sigma^{-1}x+\frac{t^2}{2}C_2\sigma^{-1}b+\frac{\sqrt{3}}6t^{\frac32}\sqrt{C_2C_2^{\top}}\right).
\end{align*}
If the initial datum $\rho_0$ of the FPK equation~\eqref{eq:FPKeqn} satisfies $\int_De^{tC_2\sigma^{-1}x}\rho_0(dx)<\infty$, then the assumption {\Ag} holds.
Thus, we have from \eqref{eq:linear-fitness} that
\[
{\cal A}g\equiv C_1:=C_2\sigma^{-1}b\in\R.
\]
Using the condition {\rm(ii)} of Lemma~{\rm\ref{lem:simpcaseRMM}}, the process $\overline{X}^y=(\overline{X}_t^y)_{t\in[0,T]}$ is given by
\begin{align*}
\overline{X}_t^y = y+ bt -\frac{\sigma C_2^{\top}}{2}t^2 +\sigma W_t,\quad (t,y)\in[0,T]\times\R^n.
\end{align*}
Hence, under the original probability measure $\Px$, we have that
\[
\overline{X}_t^y\sim N\left(y+ bt -\frac{\sigma C_2^{\top}}{2}t^2,\sigma\sigma^{\top}t\right),\quad t\in(0,T].
\]
This implies that the density function $\overline{p}(t,y;x)$ of $\overline{X}_t^y$ has the following closed-form representation given by $\overline{p}(0,y;x)=\delta_{x-y}$, and for $(t,x,y)\in(0,T]\times\R^{2n}$,
\begin{align}\label{eq:barptyxRn}
\overline{p}(t,y;x)&=\frac{1}{(2\pi)^{\frac{n}{2}}}\frac{1}{\sqrt{{\rm det}(\sigma\sigma^{\top})t}}\\
&\quad\times
\exp\left\{-\frac{1}{2t}\left(x-y-bt +\frac{\sigma C_2^{\top}}{2}t^2\right)^{\top}(\sigma\sigma^{\top})^{-1}\left(x-y-bt +\frac{\sigma C_2^{\top}}{2}t^2\right)\right\}.\nonumber
\end{align}
It follows from \eqref{eq:densityutx0} in Lemma~{\rm\ref{lem:simpcaseRMM}} that the density function of $\mu_t$ is given by
\begin{align}\label{eq:densityutx0examRn}
u(t,x)=\frac{e^{C_2\sigma^{-1}xt}\int_{\R^n}\overline{p}(t,y;x)u_0(y)dy}
{\int_{\R^n}e^{C_2\sigma^{-1}zt}\left(\int_{\R^n}\overline{p}(t,y;z)u_0(y)dy\right)dz},\quad (t,x)\in[0,T]\times\R^n.
\end{align}
This in fact establishes the probabilistic representation of the weak solution to the following RM equation with the fitness space $D=\R^n$:
\begin{align}\label{exam:1RM-eqn-Rn}
\left\{
  \begin{array}{ll}
    \displaystyle \partial_tu(t,x)=\frac{\sigma\sigma^{\top}}{2}\Delta u(t,x)-b^{\top}\nabla_xu(t,x)+\left(g(x)-\int_{\R^n}g(y)u(t,y)dy\right)u(t,x);\\ \\
    \displaystyle u(0,x)=u_0(x),~x\in\R^n.
  \end{array}
\right.
\end{align}
\end{example}
A special case of Example~\ref{eq:constdrift-volatility} above is the one-dimensional RM equation with $b\equiv0$, $\sigma=\sqrt{2}$ and the fitness function $g(x)=x$ for $x\in\R$. Then, Eq.~\eqref{exam:1RM-eqn-Rn} becomes that
\begin{align}\label{exam:ACRM-eqn-Rn}
\left\{
  \begin{array}{ll}
    \displaystyle \partial_tu(t,x)=\Delta u(t,x)+\left(x-\int_{\R}yu(t,y)dy\right)u(t,x);\\ \\
    \displaystyle u(0,x)=u_0(x),~x\in\R.
  \end{array}
\right.
\end{align}
The explicit solution to the RM equation \eqref{exam:ACRM-eqn-Rn} has been studied by \cite{alfarocarles14}. We next verify that the probabilistic solution \eqref{eq:densityutx0examRn} of the RM equation \eqref{exam:ACRM-eqn-Rn} coincides with the solution form given by (2.3) of Theorem 2.3 in \cite{alfarocarles14} via a tricky transform of the solution based on the Avron-Herbst formula. In fact, in view of Lemma~\ref{lem:simpcaseRMM}, from \eqref{eq:linear-fitness} and \eqref{eq:barptyxRn}, it results in $C_2=\sqrt{2}$, and hence
\begin{align}\label{eq:densityutx0examR1}
\overline{p}(t,y;x)=\frac{1}{2\sqrt{\pi t}}\exp\left(-\frac{(x-y+t^2)^2}{4t}\right),\quad (t,x,y)\in[0,T]\times\R^2.
\end{align}
This yields that the probability solution \eqref{eq:densityutx0examRn} of the RM equation \eqref{exam:ACRM-eqn-Rn} is given by, for $(t,x,y)\in[0,T]\times\R^2$,
\begin{align}\label{eq:densityutx011R1}
u(t,x)=\frac{e^{tx}\int_{\R}
\exp\left(-\frac{(x-y+t^2)^2}{4t}\right)u_0(y)dy}{\int_{\R}e^{tz}\left(\int_{\R}\exp\left(-\frac{(z-y+t^2)^2}{4t}\right)u_0(y)dy\right)dz}.
\end{align}
By \eqref{eq:solutionmut} and the assumption {\Ag}, it is not difficult to verify that, for $t\in[0,T)$,
\begin{align*}
  \int_{\R}xu(t,x)dx=\int_{\mathbb R}x\mu_t(dx)= \frac{\int_{D}\Ex\left[X_t^y\exp\left(\int_0^tg(X_s^y)ds\right)\right]\rho_0(dy)}{\int_{D}\Ex\left[\exp\left(\int_0^tg(X_s^y)ds\right)\right]\rho_0(dy)}<\infty.
\end{align*}
A direct calculation gives that
\begin{align*}
\int_{\R}e^{tz}\left(\int_{\R}\exp\left(-\frac{(z-y+t^2)^2}{4t}\right)u_0(y)dy\right)dz&=\int_{\R^2}\exp\left(\frac{-(z-y-t^2)^2}{4t}+ty\right)u_0(y)dzdy\notag\\
  &=\sqrt{4\pi t}\int_{\R}e^{ty}u_0(y)dy.
\end{align*}
Therefore, for $(t,x)\in[0,T)\times\R$,
\begin{align}\label{eq:sol2poit3AC}
u(t,x)&=\frac{e^{tx}\int_{\R}\exp\left(-\frac{(x-y+t^2)^2}{4t}\right)u_0(y)dy}
{\int_{\R}e^{tz}\left(\int_{\R}\exp\left(-\frac{(z-y+t^2)^2}{4t}\right)u_0(y)dy\right)dz}\nonumber\\
&=\frac{\frac{e^{tx}}{\sqrt{4\pi t}}\int_{\R}\exp\left(-\frac{(x-y+t^2)^2}{4t}\right)u_0(y)dy}{\int_{\R}e^{ty}u_0(y)dy},
\end{align}
this is the solution representation given by (2.3) of Theorem 2.3 in \cite{alfarocarles14}.

However, the condition~\eqref{eq:Agcond} in Lemma~\ref{lem:simpcaseRMM} on the coefficients $(b,\sigma,g)$ is very restrictive. In order to bypass this constraint on $(b,\sigma,g)$, we next apply a martingale extraction approach to identify the function $\Phi^f$ defined by \eqref{eq:Phitx}. More precisely, recall the operator ${\cal A}$ acted on $C^2(D)$ which is defined by \eqref{eq:operatorsLn}. Let $(\lambda,\phi)$ be an eigenpair (if exists) of the following characteristic equation given by
\begin{align}\label{eq:eigeneqn}
({\cal A}+ g(x))\phi(x) = -\lambda \phi(x),\quad x\in D.
\end{align}
Then, it follows from \eqref{eq:eigeneqn} and It\^o formula that, for $y\in D$,
\begin{align}\label{eq:Mtmart}
M_t^y := \frac{1}{\phi(y)}\exp\left(\lambda t+\int_0^tg(X_s^y)ds\right)\phi(X_t^y),\quad t\in[0,T]
\end{align}
is a local $(\Px,\Fx)$-martingale with $M_0^y=1$.
Then, by \eqref{eq:Phitx}, it holds that
\begin{align}\label{eq:Phih2}
\Phi^{f}(t,y)= e^{-\lambda t}\phi(y)\Ex\left[M_t^y\frac{f(X_t^y)}{\phi(X_t^y)}\right],\quad (t,y)\in[0,T]\times D.
\end{align}

A key observation is that the expectation in \eqref{eq:Phih2} does not depend on the path of the process $X^y$. Then, this gives that
\begin{theorem}\label{thm:mutsol3}
Let assumptions {\Aa}, {\Ab} and {\Ag} hold. Assume additionally that, for all $y\in D$,
\begin{itemize}
  \item[{\rm(i)}] there exists an eigenpair $(\lambda,\phi)$ of the characteristic equation \eqref{eq:eigeneqn} such that the process $M^y=(M_t^y)_{t\in[0,T]}$ defined by \eqref{eq:Mtmart} is a $(\Px,\Fx)$-martingale.
  \item[{\rm(ii)}] let $\overline{X}^y=(\overline{X}_t^y)_{t\in[0,T]}$ be the strong solution of the following SDE: for $t\in[0,T]$,
  \begin{align}\label{eq:SDEbarX}
    \overline{X}_t^y = y+ \int_0^t(b(\overline{X}_s^y)+\sigma(\overline{X}_s^y)\overline{\sigma}(\overline{X}_s^y)^{\top})ds + \int_0^t\sigma(\overline{X}_s^y)dW_s,
\end{align}
where the function $\overline{\sigma}(x):=\frac{\nabla_x\phi(x)^{\top}}{\phi(x)}\sigma(x)$ for $x\in D$.
\end{itemize}
Then, for any $T>0$, the solution $\mu$ given by \eqref{eq:solutionmut} of the FPK equation \eqref{eq:FPKeqn} admits the following form:
\begin{align}\label{eq:mudx2}
\mu_t(dx) =\frac{\frac{1}{\phi(x)}\int_{D}\phi(y)\Px(\overline{X}_t^y\in dx)\rho_0(dy)}{\int_{D}\frac{\int_{D}\phi(y)\Px(\overline{X}_t^y\in dz)\rho_0(dy)}{\phi(z)}},\quad t\in(0,T].
\end{align}
Moreover, if $\rho_0(dx)=u_0(x)dx$ and the density function $\overline{p}(t,y;x)$ of $\overline{X}_t^y$ for $t\in(0,T]$ exists, in other words,
\[
\Px(\overline{X}_t^y\in dx)=\overline{p}(t,y;x)dx,
\]
then, for all $t\in[0,T]$, the probability measure $\mu_t(dx)=u(t,x)dx$. Here, $u(t,x)$ satisfies the extended RM equation~\eqref{eq:weak-EM-eqn}, and it admits the following representation given by
\begin{align}\label{eq:densityutx}
u(t,x)=\frac{\frac{1}{\phi(x)}\int_{D}\phi(y)\overline{p}(t,y;x)u_0(y)dy}{\int_{D}\frac{\int_{D}\phi(y)\overline{p}(t,y;z)u_0(y)dy}{\phi(z)}dz},\quad (t,x)\in[0,T]\times D.
\end{align}
\end{theorem}

\begin{proof}
By using \eqref{eq:Mtmart} and \eqref{eq:eigeneqn}, the dynamics of the martingale $M^y$ defined by \eqref{eq:Mtmart} is given by, for $y\in D$,
\begin{align}\label{eq:martM-sde}
dM_t^y  &=M_t^y\frac{\nabla_x\phi(X_t^y)^{\top}\sigma(X_t^y)}{\phi(X_t^y)}dW_t=M_t^y\overline{\sigma}(X_t^y)dW_t,\quad M_0^y=1.
\end{align}
Define $\frac{d\Qx}{d\Px}|_{\F_t}=M_t^y$ for $t\in[0,T]$. Then, from Girsanov's theorem, it follows that $B_t:=W_t-\int_0^t\overline{\sigma}(X_s^y)^{\top}ds$ for $t\in[0,T]$ is an $m$-dimensional Brownian motion under the new probability measure $\Qx$. Then, under $\Qx$, the dynamics of the process $X^y=(X_t^y)_{t\in[0,T]}$ is given by, for $(t,y)\in[0,T]\times D$,
\begin{align*}
X_t^y = y+ \int_0^t(b(X_s^y)+\sigma(X_s^y)\overline{\sigma}(X_s^y)^{\top})ds + \int_0^t\sigma(X_s^y)dB_s.
\end{align*}
By virtue of \eqref{eq:Phih2}, we have that, for $(t,y)\in[0,T]\times D$,
\begin{align*}
\Phi^{f}(t,y)= e^{-\lambda t}\phi(y)\Ex^{\Qx}\left[\frac{f(X_t^y)}{\phi(X_t^y)}\right]=e^{-\lambda t}\phi(y)\int_{D}\frac{f(x)}{\phi(x)}\Qx(X_t^y\in dx).
\end{align*}
Apply Lemma~\ref{lem:solutionFPK}, we obtain that, for all test functions $f\in C_0^{\infty}(D)$,
\begin{align}\label{eq:solutionmut2}
\int_{D}f(x)\mu(dx) = \frac{\int_{D}\frac{f(x)}{\phi(x)}\int_{D}\phi(y)\Qx(X_t^y\in dx)\rho_0(dy)}{\int_{D}\frac{\int_{D}\phi(y)\Qx(X_t^y\in dx)\rho_0(dy)}{\phi(x)}}.
\end{align}
Note that $\Qx(X_t^y\in dx)=\Px(\overline{X}_t^y\in dx)$. Then, the desired result follows from \eqref{eq:solutionmut2}.
\end{proof}

In order to apply Theorem~\ref{thm:mutsol3}, the key point is to identify an eigenpair $(\lambda,\phi)$ of the characteristic equation \eqref{eq:eigeneqn} so to satisfy the condition (i) of Theorem~\ref{thm:mutsol3}. We next provide examples in which we illustrate how to verify this condition (i) and derive the probability solution~\eqref{eq:mudx2}-\eqref{eq:densityutx} explicitly. We first consider the one-dimensional case, i.e., $n=m=1$. In this case, let $D=(\alpha,\beta)$ with $-\infty\leq\alpha<\beta\leq+\infty$. For this purpose, for a function $V:D\to\R$ which is locally H\"older continuous, we define the following operator acted on $C^2(D)$ as:
\begin{align}\label{eq:operatorH}
{\cal H}_V := {\cal A} + V.
\end{align}
We call a positive function $\phi\in C^2(D)$ a {\it positive harmonic function} of ${\cal H}_V$ if ${\cal H}_V\phi=0$. A function $\phi\in C^2(D)$ is said to be an {\it invariant function} of the semigroup generated by ${\cal H}_V$ if $\Ex[e^{\int_0^tV(X_s^x)ds}\phi(X_t^x)]=\phi(x)$ for $(t,x)\in[0,T]\times D$.  The following result provides the necessary and sufficient condition under which a positive harmonic function of ${\cal H}_V$ on $D$ is also an invariant function of the semigroup generated by this operator.
\begin{lemma}[Theorem 5.1.8. in \cite{pinsky95}]\label{lem:Pinsky}
Let the assumption {\Aa} and {\Ab} hold and $x_0\in D$. Then, a positive harmonic function of ${\cal H}_V$ is also an invariant function of the semigroup generated by ${\cal H}_V$ if and only if the following conditions hold:
\begin{align}\label{eq:cond1deigen}
&\int_{\alpha}^{x_0}\left[\frac{1}{\phi^2(x)}\exp\left(-\int_{x_0}^x\frac{2b(z)}{\sigma^2(z)}dz\right)\int_{x}^{x_0}\frac{\phi^2(y)}{\sigma^2(y)}
\exp\left(\int_{x_0}^y\frac{2b(z)}{\sigma^2(z)}dz\right)dy\right]dx=+\infty,\nonumber\\
&\int_{x_0}^{\beta}\left[\frac{1}{\phi^2(x)}\exp\left(-\int_{x_0}^x\frac{2b(z)}{\sigma^2(z)}dz\right)\int_{x_0}^{x}\frac{\phi^2(y)}{\sigma^2(y)}
\exp\left(\int_{x_0}^y\frac{2b(z)}{\sigma^2(z)}dz\right)dy\right]dx=+\infty.
\end{align}
\end{lemma}
Lemma~\ref{lem:Pinsky} can be used to verify the martingale property of $M^y=(M_t^y)_{t\in[0,T]}$ defined by \eqref{eq:Mtmart} in some one-dimensional cases. We next provide examples in the one-dimensional case by applying Lemma~\ref{lem:Pinsky}.
\begin{example}\label{exam:one-dimOU}
Consider $m=n=1$, $b(x)=\kappa(\theta-x)$, and $\sigma(x)=\sigma$ for $x\in D:=\R$, where $\theta,\kappa\in\R$ with $\kappa\neq0$ and $\sigma>0$. Then, assumptions {\Aa} and {\Ab} are satisfied. In this case, the process $X^{t,x}=(X_s^{t,x})_{s\in[t,T]}$ given by \eqref{eq:SDEXtx} satisfies that, for all $(t,x)\in[0,T]\times\R$,
\begin{align}\label{eq:OU}
dX_s^{t,x} = \kappa(\theta-X_s^{t,x})ds + \sigma dW_s,\quad X_t^{t,x}=x\in\R.
\end{align}
The fitness function is given by $g(x)=-x$ for $x\in\R$. We next compute $\Ex[\exp(-\int_0^tX_s^{0,x}ds)]$ for $(t,x)\in[0,T]\times\R$. In fact, we have that
\begin{align*}
\int_0^tX_s^{0,x}ds&=(1-e^{-\kappa t})(x-\kappa^{-1}\theta)+\kappa\theta t+\sigma W_t-\sigma e^{-\kappa t}\int_0^te^{\kappa s}dW_s.
\end{align*}
This yields that $\int_0^tX_s^{0,x}ds\sim N({\mu}(t,x),\sigma^2(t))$ for $t\in(0,T]$, where $\mu(t,x):=(1-e^{-\kappa t})(x-\kappa^{-1}\theta)+\kappa\theta t$ and
$\sigma^2(t):=\sigma^2t+\frac{\sigma^2}{2\kappa}(1-e^{-2\kappa t})-\frac{2\sigma^2}{\kappa}(1-e^{-\kappa t})$. Therefore, for $t\in[0,T]$,
\begin{align}\label{eq:expouint}
\Ex\left[\exp\left(-\int_0^tX_s^{0,x}ds\right)\right]&=e^{-\mu(t,x)+\frac{1}{2}\sigma^2(t)}.
\end{align}
Moreover, it results in
\begin{align*}
\int_{\R}\Ex\left[\exp\left(-\int_0^tX_s^{0,x}ds\right)\right]\rho_0(dx)=e^{\frac{\sigma^2(t)}{2}+(1-e^{-\kappa t})\kappa^{-1}\theta-\kappa\theta t}\int_{\R}e^{-(1-e^{-\kappa t})x}\rho_0(dx).
\end{align*}
This implies that if the initial data $\rho_0$ of the FPK equation \eqref{eq:FPKeqn} satisfies that
\begin{align}\label{eq:condrho0ou}
\int_{\R}e^{-(1-e^{-\kappa t})x}\rho_0(dx)<\infty,
\end{align}
then the  assumption {\Ag} is satisfied.

Define $V(x):=g(x)+\lambda=\lambda-x$ for $x\in\R$. Thus, the characterize equation \eqref{eq:eigeneqn} becomes ${\cal H}_V\phi=0$. For $\lambda\in\R$, let us define
\begin{align}\label{eq:mulambda}
\mu(\lambda):=\left\{
  \begin{array}{ll}
    \displaystyle \frac1\kappa\left(\lambda-\theta+\frac{\sigma^2}{2\kappa^2}\right), & \kappa>0; \\ \\
    \displaystyle \frac1{|\kappa|}\left(\lambda-\theta+\frac{\sigma^2}{2\kappa^2}+\kappa\right), & \kappa<0.
  \end{array}
\right.
\end{align}
We take $\lambda$ so that $\mu(\lambda)\leq0$. Then, together with Lemma~{\rm\ref{lem:Pinsky}}, by applying Proposition {\rm 6.2} and Theorem {\rm 6.2} in \cite{qinLinetsky16}, the associated semigroup generated by ${\cal H}_V$ (depending on $\lambda$) has invariant functions given by:
\begin{itemize}
  \item[{\rm(i)}] for $\mu(\lambda)<0$, the invariant functions are of the form:
\begin{align}\label{eq:invariantfcn1}
  \phi_\lambda(x)=C_1e^{\frac{z(x)^2\epsilon}4}E_\mu(z(x)-\alpha)+C_2e^{\frac{z(x)^2\epsilon}4}E_\mu(\alpha-z(x)),\quad C_1,C_2>0,
\end{align}
where $\alpha:=\sigma\sqrt{\frac2{|\kappa|^3}}$, $z(x):=\frac{\sqrt{2|\kappa|}}\sigma(\theta-x)$, and $\epsilon:=\text{\rm sign}(\kappa)$.
\item[{\rm(ii)}] for $\mu(\lambda)=0$, the invariant functions are of the form:
\begin{align}
  \phi_\lambda(x)=C_1e^{\frac{z(x)^2\epsilon}4}E_\mu(z(x)-\alpha),\quad C_1>0.
\end{align}
\end{itemize}
Here, $E_\mu(\cdot)$ denotes the Weber parabolic cylinder function. Noting that the following recurrence relations on the Weber parabolic cylinder function hold:
\begin{align*}
E_{\mu+1}(z)-zE_{\mu}(z)+\mu E_{\mu-1}(z)=0,\quad E'_\mu(z)+\frac12zE_\mu(z)-\mu E_{\mu-1}(z)=0.
\end{align*}
Then, we arrive at the following relations given by
\begin{align}\label{Eq:Emurelations}
\begin{cases}
\displaystyle \frac d{dz}\left(e^{\frac {z^2}4}E_{\mu}(z-\alpha)\right)=\frac\alpha2e^{\frac{z^2}4}E_\mu(z-\alpha)+\mu e^{\frac{z^2}4}E_{\mu-1}(z-\alpha);\\ \\
\displaystyle \frac d{dz}\left(e^{\frac {z^2}4}E_{\mu}(\alpha-z)\right)=\frac\alpha2e^{\frac{z^2}4}E_\mu(\alpha-z)-\mu e^{\frac{z^2}4}E_{\mu-1}(z-\alpha);\\ \\
\displaystyle \frac d{dz}\left(e^{-\frac {z^2}4}E_{\mu}(z-\alpha)\right)=-\frac\alpha2e^{-\frac {z^2}4}E_{\mu}(z-\alpha)-e^{\frac {z^2}4}E_{\mu+1}(z-\alpha);\\ \\
\displaystyle \frac d{dz}\left(e^{-\frac {z^2}4}E_{\mu}(\alpha-z)\right)=-\frac\alpha2e^{-\frac {z^2}4}E_{\mu}(\alpha-z)+e^{\frac {z^2}4}E_{\mu+1}(\alpha-z).
\end{cases}
\end{align}
Therefore, the dynamic of $\overline{X}^y_t$ given in the condition (ii) of Theorem~\ref{thm:mutsol3} are respectively as follows in terms of the sign of $\kappa$:
\begin{align*}
\begin{cases}
\displaystyle d\overline{X}_t^y=b^{+}(\overline{X}_t^y)+\sigma dW_t, & \kappa>0;\\ \\
\displaystyle d\overline{X}_t^y=b^{-}(\overline{X}_t^y)+\sigma dW_t, & \kappa<0.
\end{cases}
\end{align*}
The drift coefficient functions in the above SDEs are given by:
\begin{align*}
\overline{b}^{+}(x)&:=\kappa\theta-\frac{\sigma^2}\kappa-\kappa x+\sqrt{2|\kappa|}\mu\sigma\frac{C_2E_{\mu-1}
  \left(\alpha-z\left(x\right)\right)-C_1E_{\mu-1}\left(z(x)-\alpha\right)}
  {C_2E_\mu\left(\alpha-z\left(x\right)\right)+C_1E_\mu\left(\alpha-z\left(x\right)\right)};\\
\overline{b}^{-}(x)&:=\kappa\theta+\frac{\sigma^2}\kappa-\kappa x+\sqrt{2|\kappa|}\sigma
  \frac{C_1E_{\mu+1}\left(z(x)-\alpha\right)-C_2E_{\mu+1}\left(\alpha-z(x)\right)}
  {C_2E_\mu\left(\alpha-z(x)\right)+C_1E_\mu\left(\alpha-z(x)\right)}.
\end{align*}
Define $W^y_t:=\sigma^{-1}y+W_t$, $N_{t}^{y,\pm}:=\exp(\int_0^t\sigma^{-1}b^{\pm}(\sigma W^y_s)dW^y_s-\frac{1}{2\sigma^2}\int_0^tb^{\pm}(\sigma W^y_s)^2ds)$ and $\overline{Y}_t^y:=\sigma^{-1}\overline{X}_t^y$ for $t\in[0,T]$. Using Exercise 5.5.38 in \cite{karatzasshreve1991}, it follows that, for all $t\in(0,T]$ and $z\in\R$,
\begin{align*}
\Px(\overline{Y}_t^{y}\leq z)=\Ex\left[N_t^{y,{\rm sgn}(\kappa)}{\I}_{W^y_t\leq z}\right]=\int_{-\infty}^{z}\Ex\left[N_{t}^{y,{\rm sgn}(\kappa)}\Big|W^y_t=r\right]p_0\left(t,\frac{y}{\sigma};r\right)dr,
\end{align*}
where $p_0(t,y;x):=\frac{1}{\sqrt{2\pi t}}\exp(-\frac{(x-y)^2}{2t})$ for $(t,x,y)\in(0,T]\times\R^2$. Therefore, for all $(t,x,y)\in(0,T]\times\R^2$,
\begin{align*}
\Px(\overline{X}_t^{y}\leq x)&=\Px\left(\overline{Y}_t^{y}\leq \frac{x}{\sigma}\right)
=\int_{-\infty}^{\frac{x}{\sigma}}\Phi^{{\rm sgn}(\kappa)}(t,r,y)p_0\left(t,\frac{y}{\sigma};r\right)dr.
\end{align*}
Here $\Phi^{\pm}(t,r,y):=\Ex[N_{t}^{y,\pm}|W^y_t=r]$ for all $(t,r,y)\in(0,T]\times\R^2$. Then, the density function $\overline{p}(t,y;x)$ of $\overline{X}_t^{y}$ for $t\in(0,T]$ is given by, for $(t,x,y)\in(0,T]\times\R^2$,
\begin{align}\label{eq:barpxytexamOU}
\overline{p}(t,y;x)=\frac{1}{\sigma}\Phi^{{\rm sgn}(\kappa)}\left(t,\frac{x}{\sigma},y\right)p_0\left(t,\frac{y}{\sigma};\frac{x}{\sigma}\right).
\end{align}
Then, by the solution representation \eqref{eq:densityutx} in Theorem \ref{thm:mutsol3}, and \eqref{eq:condrho0ou}, for any probability density function $u_0$ satisfying $\int_{\R}e^{-(1-e^{-\kappa t})x}u_0(x)dx<\infty$, we have the solution of the extended RM equation \eqref{eq:weak-EM-eqn} given by
\begin{align*}
u(t,x)=\frac{\frac{1}{\phi_{\lambda}(x)}\int_{-\infty}^{\infty}\phi_{\lambda}(y)\overline{p}(t,y;x)u_0(y)dy}
{\int_{-\infty}^{\infty}\frac{\int_{-\infty}^{\infty}\phi(y)\overline{p}(t,y;z)u_0(y)dy}{\phi_{\lambda}(z)}dz},\quad (t,x)\in[0,T]\times\R.
\end{align*}
\end{example}

\begin{example}\label{exam:one-dimCIR}
Consider $m=n=1$, $b(x)=a+bx$, and $\sigma(x)=\sigma\sqrt{x}$ for $x\in D:=(0,\infty)$, where $a,\sigma>0$ and $b\in\R$. Let the Feller condition hold, i.e., $2a\geq\sigma^2$. Then, assumptions {\Aa} and {\Ab} are satisfied. In this case, the process $X^{t,x}=(X_s^{t,x})_{s\in[t,T]}$ given by \eqref{eq:SDEXtx} satisfies that, for all $(t,x)\in[0,T]\times D$,
\begin{align}\label{eq:CIR}
dX_s^{t,x} = (a+bX_s^{t,x})ds + \sigma\sqrt{X_s^{t,x}}dW_s,\quad X_t^{t,x}=x\in D.
\end{align}
The fitness function is given by $g(x)=-x$ for $x\in D$. Then, the assumption {\Ag} holds. Define $V(x):=g(x)+\lambda=\lambda-x$ for $x\in D$. Thus, the characterize equation \eqref{eq:eigeneqn} becomes ${\cal H}_V\phi=0$. Together with Lemma~{\rm\ref{lem:Pinsky}}, by applying Proposition {\rm 6.1} and Theorem {\rm 6.1} in \cite{qinLinetsky16}, for $\lambda\leq\lambda_0:=\frac{a(\sqrt{b^2+2\sigma^2}+b)}{\sigma^2}$, the semigroup generated by ${\mathcal H}_V$ (depending on $\lambda$) has positive invariant functions $\phi_\lambda$ with the form given by, for $x>0$,
\begin{align}\label{eq:invKummer}
  \phi_\lambda(x)=C_1e^{\frac{\kappa-\gamma}{\sigma^2}x}K\left(\alpha,\beta,\frac{2\gamma}{\sigma^2}x\right),\quad C_1>0.
\end{align}
Here $\kappa=-b$, $\gamma=\sqrt{\kappa^2+2\sigma^2}$, $\alpha=\frac{\lambda-\lambda_0}\gamma$, $\beta=\frac{2a}{\sigma^2}$, and $K(\cdot,\cdot,\cdot)$ is the Kummer confluent hypergeometric function. By plugging \eqref{eq:invKummer} into \eqref{eq:SDEbarX} and \eqref{eq:densityutx}, we obtain the dynamic of $\overline{X}^y_t$ given in the condition (ii) of Theorem~\ref{thm:mutsol3} as follows:
\begin{align}\label{eq:barXyCIR}
  d\overline{X}_t^y=\left\{a-\gamma \overline{X}^y_t+\frac{2\alpha\gamma}\beta\frac{K\left(\alpha+1,\beta+1,\frac{2\gamma}{\sigma^2}\overline{X}^y_t\right)}
  {K\left(\alpha,\beta,\frac{2\gamma}{\sigma^2}\overline{X}^y_t\right)}\overline{X}^y_t\right\}dt+\sigma\sqrt{\overline{X}^y_t}dW_t.
\end{align}
For $(t,x,y)\in[0,T]\times(0,\infty)^2$, let $\overline{p}(t,y;x)$ be the transition density function of $\overline{X}_t^y$ (note that $\overline{p}(t,y;x)$ does not admit a closed-form in general). Then, by \eqref{eq:densityutx} of Theorem \ref{thm:mutsol3}, for any probability density function $u_0(x)$ on $x\in(0,\infty)$, we have the solution of the extended RM equation \eqref{eq:weak-EM-eqn} given by
\begin{align*}
u(t,x)=\frac{\frac{1}{\phi_{\lambda}(x)}\int_{0}^{\infty}\phi_{\lambda}(y)\overline{p}(t,y;x)u_0(y)dy}{\int_{0}^{\infty}
\frac{\int_{0}^{\infty}\phi_{\lambda}(y)\overline{p}(t,y;z)u_0(y)dy}{\phi_{\lambda}(z)}dz},\quad (t,x)\in[0,T]\times(0,\infty).
\end{align*}
\end{example}

\begin{example}\label{exam:polynomialcff}
Consider $m=n=1$, $b(x)\equiv0$ and $\sigma(x)\equiv\sqrt{2}\sigma$ for $x\in D=\R$, where $\sigma>0$. Then, assumptions {\Aa} and {\Ab} are satisfied. In this case, the process $X^{t,x}=(X_s^{t,x})_{s\in[t,T]}$ given by \eqref{eq:SDEXtx} satisfies that, for all $(t,x)\in[0,T]\times\R$,
\begin{align}\label{eq:SDEXpoly}
X_s^{t,x} = y + \sqrt{2}\sigma (W_s-W_t),\quad s\in[t,T].
\end{align}
Here, we consider the polynomial confining fitness function of degree $2q$ which is introduced in \cite{alfaroveruete19}, i.e., for a given $q\in\N$,
\begin{align}\label{eq:polynomialconfitfcn}
g(x) = -x^{2q} + \sum_{l=0}^{2q-1}\alpha_lx^l,\quad x\in\R,
\end{align}
where the sequence of constants $\alpha_l$ with $l=0,\ldots,2q-1$. Then, the resulting equation~\eqref{eq:RM-eqn-Rn} becomes the integro-differential Cauchy problem studied in \cite{alfaroveruete19}. Note that the operator ${\cal A}$ acted on $C^2(\R)$ is given by
\[
{\cal A}f(x)=\sigma^2f''(x),\quad x\in\R.
\]
Then, the eigenpair $(\lambda,\phi)$ of the characteristic equation \eqref{eq:eigeneqn} is related to the spectral basis of Schr\"odinger operator $-{\cal A}-g$ with confining potential $-g$. By virtue of \eqref{eq:polynomialconfitfcn}, it can be seen that $g:\R\to\R$ is continuous and $\lim_{|x|\to\infty}g(x)=-\infty$. Then, $-{\cal A}-g$ has discrete spectrum, i.e., there is an orthonormal basis $(\phi_k)_{k\geq1}\subset L^2(\R)$ such that $({\cal A}+g)\phi_k=-\lambda_k\phi_k$ and~$\|\phi_k\|_{L^2(\R)}=1$, with corresponding eigenvalues
\[
\lambda_0<\lambda_1\leq\lambda_2\leq\cdots\leq\lambda_k\to\infty
\]
of finite multiplicity. Moreover, Proposition~{\rm 2.6} in \cite{alfaroveruete19} gives that
\[
\|\phi_k\|_{L^{\infty}(\R)}\leq Ck^{\frac{q}{2(q+1)}},\quad k\geq1.
\]
Up to subtracting a constant to $g$, we can assume without loss of generality that $g\leq0$. This yields the assumption {\Ag}. Recall the local martingale $M^y$ defined by \eqref{eq:Mtmart} with the eigenpair $(\lambda_k,\phi_k)$ for $k\geq1$. Then, it holds that, $M_0^y=1$ and for $t\in(0,T]$,
\begin{align*}
  dM_t^{k,y} =\frac{\sqrt2\sigma}{\phi_k(y)}\exp\left(\lambda t+\int_0^tg(X_s^y)ds\right)\phi'_k(X_t^y)dW_t.
\end{align*}
Note that, the normalized eigenfunction $\phi_k$ for $k\geq1$ satisfies that
\begin{align}
  \sigma^2\int_{\R}|\phi'_k(x)|^2dx-\int_{\R}g(x)|\phi_k(x)|^2dx=\lambda_k,
\end{align}
this yields that $\phi'_k(\pm\infty)=0$. Therefore $\phi'_k$ for $k\geq1$ is bounded on $\R$. This implies that the local martingale $M^{k,y}$ defined by \eqref{eq:Mtmart} with the eigenpair $(\lambda_k,\phi_k)$ is a $(\Px,\Fx)$-martingale. In other words, the condition (i) of Theorem~\ref{thm:mutsol3} holds.

For $k\geq1$ and $y\in\R$, consider the following SDE given by
\begin{align}\label{eq:SDEbarXk}
 \overline{Y}_t^{k,y} = \frac{y}{\sqrt{2}\sigma}+ \int_0^t\overline{\sigma}_k(\sqrt{2}\sigma\overline{Y}_s^{k,y})ds + W_t,\quad t\in[0,T],
\end{align}
where $\overline{\sigma}_k(x):=\sqrt{2}\frac{\phi_k'(x)}{\phi_k(x)}\sigma$ for $x\in\R$. Then, the process $\overline{X}^{k,y}=(\overline{X}^{k,y}_t)_{t\in[0,T]}$ given in Theorem~\ref{thm:mutsol3}-(ii) is given by
\[
\overline{X}_t^{k,y}=\sqrt{2}\sigma\overline{Y}_t^{k,y},\quad t\in[0,T].
\]
By applying Exercise 5.5.38 in \cite{karatzasshreve1991}, we arrive at, for all $t\in(0,T]$ and $z\in\R$,
\begin{align*}
\Px(\overline{Y}_t^{k,y}\leq z)=\Ex\left[N_t^{k,y}{\I}_{W^y_t\leq z}\right]=\int_{-\infty}^{z}\Ex\left[N_t^{k,y}\Big|W^y_t=r\right]p_0\left(t,\frac{y}{\sqrt{2}\sigma};r\right)dr,
\end{align*}
where $W^y_t=\frac{y}{\sqrt{2}\sigma}+W_t$ and $N_t^{k,y}:=\exp(\int_0^t\overline{\sigma}_k(\sqrt{2}\sigma W^y_s)dW^y_s-\frac{1}{2}\int_0^t\overline{\sigma}_k^2(\sqrt{2}\sigma W^y_s)ds)$ for $(t,y)\in[0,T]\times{\mathbb R}$. Here $p_0(t,y;x):=\frac{1}{\sqrt{2\pi t}}\exp(-\frac{(x-y)^2}{2t})$ for $(t,x,y)\in(0,T]\times\R^2$. Therefore, for all $(t,x,y)\in(0,T]\times\R^2$,
\begin{align*}
\Px(\overline{X}_t^{k,y}\leq x)&=\Px\left(\overline{Y}_t^{k,y}\leq \frac{x}{\sqrt{2}\sigma}\right)
=\int_{-\infty}^{\frac{x}{\sqrt{2}\sigma}}\Phi_k(t,r,y)p_0\left(t,\frac{y}{\sqrt{2}\sigma};r\right)dr.
\end{align*}
Here $\Phi_k(t,r,y):=\Ex[N_t^{k,y}|W^y_t=r]$ for all $(t,r,y)\in(0,T]\times\R^2$. Then, the density function $\overline{p}_k(t,y;x)$ of $\overline{X}_t^{k,y}$ for $t\in(0,T]$ exists, and it is given by, for $(t,x,y)\in(0,T]\times\R^2$,
\begin{align}\label{eq:barpxytexam}
\overline{p}_k(t,y;x)=\frac{1}{\sqrt{2}\sigma}\Phi_k\left(t,\frac{x}{\sqrt{2}\sigma},y\right)p_0\left(t,\frac{y}{\sqrt{2}\sigma};\frac{x}{\sqrt{2}\sigma}\right).
\end{align}
Then, for $k\geq1$, the function defined below
\begin{align}\label{eq:densityutxexam}
u_k(t,x)=\frac{\frac{1}{\phi_k(x)}\int_{\R}\phi_k(y)\overline{p}_k(t,y;x)u_0(y)dy}{\int_{\R}\frac{\int_{\R}\phi_k(y)\overline{p}_k(t,y;z)u_0(y)dy}{\phi_k(z)}dz},\quad (t,x)\in[0,T]\times\R
\end{align}
is a (weak) solution of of the extended RM equation~\eqref{eq:weak-EM-eqn}. If the uniqueness of solutions of the extended RM equation~\eqref{eq:weak-EM-eqn} holds, then $u_k=u_l$ for $k\neq l$. The uniqueness of solutions of the extended RM equation~\eqref{eq:weak-EM-eqn} will be implied by the propagation of chaos which will be established in Section~\ref{sec:propagation}.
\end{example}

The example with multi-dimensional fitness space and quadratic fitness function is provided below:
\begin{example}\label{exam:affineprocess}
Consider $m=n\geq1$, $b(x)=b+ Bx$, and $\sigma(x)\equiv\sigma$ for $x\in D:=\R^n$, where $b\in\R^{n},B\in\R^{n\times n}$ and $\sigma\in\R^{n\times n}$ which is an invertible matrix, so that $a:=\sigma\sigma^{\top}$ is strictly positive definite. The fitness function is given by $g(x)=-r(x)$ for $x\in\R^n$, and
\begin{align}\label{eq:fitness-g-2}
r(x)=\alpha+ \delta^{\top}x+ x^{\top}Gx,\quad x\in\R^n,
\end{align}
where $\alpha\in\R$, $\delta\in\R^{n}$, and $G\in\R^{n\times n}$ is a symmetric positive semi-definite matrix. They are taken to be such that $r(x)$ is non-negative for all $x\in\R^n$. Then, the assumption {\Ag} holds. Let $V=g(x)+\lambda=\lambda-r(x)$ for $x\in\R^n$, then the characteristic equation \eqref{eq:eigeneqn} becomes ${\cal H}_V\phi=0$.
Let $H\in\R^{n\times n}$ satisfies the following continuous-time algebraic Riccati equation:
\begin{align}\label{eq:CTARE}
2HaH-B^{\top}H-HB-G=0,
\end{align}
and $v\in\R^n$ satisfies the following linear equation:
\begin{align}\label{eq:u-lineareqn}
2HaH -B^{\top}v-2Hb-\delta=0.
\end{align}
By Section 6.2.2 of \cite{qinLinetsky16}, the eigenpair $(\lambda,\phi)$ of the characteristic equation \eqref{eq:eigeneqn} which satisfies the condition (i) of Theorem \ref{thm:mutsol3} is given by
\begin{align}\label{eq:eigenpairexam}
(\lambda,\phi(x))=\left(\alpha-\frac{1}{2}v^{\top}av+{\rm tr}(aV)+v^{\top}b,\exp\left(-v^{\top}x-x^{\top}Hx\right)\right),\quad x\in\R^n,
\end{align}
where ${\rm tr}$ denotes the trance operator. Define $\overline{\sigma}(x)=-(v+2Hx)^{\top}\sigma$ for $x\in\R^n$. Therefore, the process $\overline{X}^y=(\overline{X}_t^y)_{t\in[0,T]}$ given in the condition (ii) of Theorem \ref{thm:mutsol3} satisfies that, for $y\in\R^n$,
\begin{align}\label{eq:SDEbarXexam}
    \overline{X}_t^y = y+ \int_0^t(b+B\overline{X}_s^y+\sigma\overline{\sigma}(\overline{X}_s^y)^{\top})ds + \sigma W_t.
\end{align}
Note that the process $\overline{X}^y$ is an $n$-dimensional affine process and hence the transition density function $\overline{p}(t,y;x)$ of $\overline{X}^y$ exists and the characteristic function of $\overline{p}(t,y;\cdot)$ has exponential-affine dependence on $y\in\R^n$, see Theorem 2.7 of \cite{duffieaffineaap}. In fact, we can rewrite the process $\overline{X}^y$ described as \eqref{eq:SDEbarXexam} as in the form given by, for $y\in\R^n$,
\begin{align*}
  \overline{X}_t^y=e^{\Gamma t}y+\int_0^te^{\Gamma(t-s)}\beta ds+\int_0^te^{\Gamma(t-s)}\sigma dW_s,\quad t\in[0,T],
\end{align*}
where $\beta:=b-\sigma\sigma^\top v$ and $\Gamma=B-2\sigma\sigma^\top H$. This yields that, for $t\in(0,T]$,
\begin{align*}
\overline{X}^y_t\sim N\left(\mu_t,\Sigma_t\right),
\end{align*}
where $\mu_t:=e^{\Gamma t}y+\int_0^te^{\Gamma(t-s)}\beta ds$ and $\Sigma_t:=\int_0^te^{\Gamma(t-s)}\sigma\sigma^\top(e^{\Gamma(t-s)})^\top ds$. Therefore, the transition density function of $\overline{X}_t^y$ is given by: for $(t,x,y)\in(0,T]\times\R^{2n}$,
\begin{align}\label{Eq:affineGaussian}
\overline{p}(t,y;x)=\frac1{\sqrt{(2\pi)^n{\rm det}(\Sigma_t)}}\exp\left(-\frac12(x-\mu_t)^\top\Sigma_t^{-1}(x-\mu_t)\right).
\end{align}
Then, by \eqref{eq:densityutx} of Theorem \ref{thm:mutsol3}, for any probability density function $u_0(x)$ on $x\in\R^n$, we have the solution of the extended RM equation \eqref{eq:weak-EM-eqn} given by
\begin{align*}
u(t,x)=\frac{\frac{1}{\phi(x)}\int_{\R^n}\phi(y)\overline{p}(t,y;x)u_0(y)dy}{\int_{\R^n}\frac{\int_{\R^n}\phi(y)\overline{p}(t,y;z)u_0(y)dy}{\phi(z)}dz},\quad (t,x)\in[0,T]\times\R^n.
\end{align*}
\end{example}

\section{Propagation of chaos}\label{sec:propagation}

The aim of this section is to establish the propagation of chaos of the FPK equation \eqref{eq:FPKeqn}. It is equivalent to the convergence of the empirical measure of a particle system to an arbitrary solution to the FPK equation \eqref{eq:FPKeqn} with respect to a suitable distance adapted to our probabilistic framework.

\subsection{Particle system}\label{sec:particlesys}

Building upon the probabilistic solution form of the extended RM equation discussed in Section~\ref{sec:FPKequationexample}, we construct a suitable coupling between a particle system and the realization of the solution to the FPK equation \eqref{eq:FPKeqn}.

Before introducing the state dynamics of the particle system proposed in this section, we first impose the following assumption on the fitness function used in this section:
\begin{description}
  \item[{\Agl}] (i) the fitness function $g:D\to\R$ is continuous and bounded from above; (ii) there exists a polynomial $Q_g:\R_+\to\R_+$ such that
  \begin{align}\label{eq:assumg-poly}
  |g(x)-g(y)|\leq Q_g(|x|+|y|)|x-y|,\quad {\rm for~all}~x,y\in D.
  \end{align}
\end{description}
Note that the condition (i) on the boundedness from above of $g$ in the assumption {\Agl} implies the assumption {\Ag}. It can be seen that the fitness functions in the examples presented in Section~\ref{sec:FPKequationexample} satisfy the condition \eqref{eq:assumg-poly} in the assumption {\Agl}. We next describe the particle system. More precisely, applying the solution representation \eqref{eq:solutionmut} of the FPK equation given in Lemma~\ref{lem:solutionFPK}, we introduce the following particle system as follows: let the number of particles be given by $N\geq1$. For $i=1,\ldots,N$, the dynamics of the state process of the $i$-th particle is given by, for $t\in[0,T]$,
\begin{align}\label{eq:particlesXi}
dX_t^i = b(X_t^i)dt + \sigma(X_t^i)dW_t^i,\quad X_0^i\in D,
\end{align}
where $W^i=(W_t^i)_{t\in[0,T]}$, $i=1,\ldots,N$ and $W=(W_t)_{t\in[0,T]}$ are independent ($m$-dimensional) Brownian motions under the filtered probability space $(\Omega,\F,\Fx,\Px)$. Here, the initial values $(X_0^i)_{i\geq1}$ is assumed to satisfy that
\begin{description}
  \item[{\Ax}] For $q\geq2$, the sequence of r.v.s $(X_0^i)_{i\geq1}$ is i.i.d. according to the probability distribution $\rho_0\in{\cal P}_{({\rm deg}(Q_g)+1)q}(D)$, where ${\rm deg}(Q_g)$ denotes the degree of the polynomial $Q_g$ which is given in the assumption {\Agl}.
\end{description}
We also strengthen the assumption on the coefficients $(b,\sigma)$ of SDE~\eqref{eq:SDEXtx} imposed in {\Aa} used in the previous section as:
\begin{description}
  \item[{\Absig}]$b:D\to\R^{n\times1}$ and $\sigma:D\to\R^{n\times m}$ are Lipschitiz continuous with linear growth.
\end{description}
The current assumption {\Absig} can yield the assumption {\Ab}, and moreover, under the assumption {\Ax}, for $q\geq2$, there exists a constant $C=C_{g,q,b,\sigma,T}>0$ such that
\begin{align}\label{eq:momentqq}
\Ex\left[\sup_{t\in[0,T]}\left|X_t^i\right|^{({\rm deg}(Q_g)+1)q}\right]\leq C\left\{1+\Ex\left[|X_0^i|^{({\rm deg}(Q_g)+1)q}\right]\right\}<\infty,
\end{align}
see also Chapter 5 in \cite{karatzasshreve1991}.

Motivated from the solution representation \eqref{eq:solutionmut} of the FPK equation \eqref{eq:FPKeqn} given in Lemma~\ref{lem:solutionFPK}, we introduce the following sequence of ${\cal P}(D)$-valued processes defined by, for $N\geq1$, and $t\in[0,T]$,
\begin{align}\label{eq:solutionmutN}
\mu_t^N = \frac{\frac{1}{N}\sum_{i=1}^N\delta_{X_t^i}\exp\left(\int_0^tg(X_s^i)ds\right)}{\frac{1}{N}\sum_{i=1}^N\exp\left(\int_0^tg(X_s^i)ds\right)},\qquad {\rm on}~{\cal B}_{D}.
\end{align}
Similarly to \eqref{eq:limitmu}, for $N\geq1$, the ${\cal P}(D)$-valued process $\mu^N=(\mu_t^N)_{t\in[0,T]}$ also satisfies that, for $t\in[0,T]$,
\begin{align}\label{eq:empiricalmeasure}
\mu_t^N= \frac{1}{N}\sum_{i=1}^N\delta_{X_t^i}\exp\left(\int_0^t(g(X_s^i)-\lc\mu_s^N,g\rc)ds\right),\quad {\rm on}~{\cal B}_{D}.
\end{align}
From the representation \eqref{eq:empiricalmeasure}, $(\mu^N)_{N\geq1}$ is in fact a sequence of empirical measure-valued processes with an exponential weight.

\subsection{Metric between $\mu^N$ and $\mu$}

This section introduces an appropriate metric between $\mu^N$ and $\mu$ for establishing the propagation of chaos. We first give the following spaces. Let ${\cal P}^s(D)$ be the set of sub-probability measures on $\BRn$, i.e., for any $\mu\in{\cal P}^s(D)$, $\mu$ is a finite measure on $\BRn$ such that $\mu(D)\leq1$. We next introduce the (Alexandroff) one-point compactification. Add one point which is outside of $D$ to $D$ called ``$\star$'' and define $D_{\star}:=D\cup\{\star\}$. Let $D$ be topologized by a topology ${\cal T}$, and we then can define a topology ${\cal T}^{\star}$ for $D_{\star}$ as follows: (i) each open subset of $D$ is also in ${\cal T}^{\star}$, i.e., ${\cal T}\subset{\cal T}^{\star}$; (ii) for each compact set $C\subset D$, define an element $U_C\in{\cal T}^{\star}$ by $U_C:=(D\setminus C)\cup\{\star\}$. Let us define a bijection $\iota:{\cal P}^s(D)\to{\cal P}(D_{\star})$ as: for any $A\in{\cal B}(D_{\star})$,
\begin{align}\label{eq:iotamap}
(\iota\mu)(A) := \mu(A\cap D)+(1-\mu(D))\delta_{\star}(A),\quad \mu\in{\cal P}^s(D).
\end{align}
Then, the integral of $\mu\in{\cal P}^s(D)$ w.r.t. a measurable function $f:D_{\star}\to\R$ is defined as (if it is well-defined):
\begin{align}\label{eq:integral-star}
\int_{D_{\star}}f(x)(\iota\mu)(dx)=\int_{D}f(x)\mu(dx)+f(\star)(1-\mu(D))=\lc\mu,f\rc+f(\star)(1-\mu(D)).
\end{align}

Consider $\mu=(\mu_{t})_{t\in[0,T]}$ as an arbitrary ${\cal P}(D)$-valued solution of the FPK equation \eqref{eq:FPKeqn}, and we then define that
\begin{align}\label{eq:mutg}
\mu_t^g:=\exp\left(\int_0^t\lc\mu_s,g\rc ds\right)\mu_t,\quad t\in[0,T].
\end{align}
This implies that, for any test function $f\in C_0^{\infty}(D)$,
\begin{align}\label{eq:discountFPK}
\lc\mu_t^g,f\rc=\lc\rho_0,f\rc+\int_0^t\lc\mu_s^g,({\cal A}+g)f\rc ds,\quad t\in[0,T].
\end{align}
For the sequence of empirical measure-valued processes $(\mu^N)_{N\geq1}$ defined by \eqref{eq:solutionmutN}, similarly to \eqref{eq:mutg}, we can also define
\begin{align}\label{eq:mutgN}
\mu_t^{g,N}:=\exp\left(\int_0^t\lc\mu_s^N,g\rc ds\right)\mu_t^N,\quad t\in[0,T].
\end{align}
Hereafter, without loss of generality, we assume that $g(x)\leq0$ for all $x\in D$ by the assumption {\Agl}. Then, for any $t\in(0,T]$, it follows from \eqref{eq:mutg} and \eqref{eq:mutgN} that $\mu_t^g$ and $\mu_t^{g,N}$ are sub-probability measures on $\BRn$, i.e., $\mu_t^g,\mu_t^{g,N}\in{\cal P}^s(D)$ (note that $\mu_0^{g,N}=\mu^N$ and $\mu_0^g=\mu_0$ and hence $\mu_0^{g,N},\mu_0^g\in{\cal P}(D)$).

For the parameter $q\geq2$ given in the assumption {\Ax}, we next establish a metric $d_{q,T}$ between $\mu=(\mu_t)_{t\in[0,T]}$ and $\mu^N=(\mu_t^N)_{t\in[0,T]}$ given as follows:
\begin{align}\label{eq:dp}
d_{q,T}(\mu,\mu^N):= \left\{\Ex\left[\sup_{t\in[0,T]}d_{BL}(\iota\mu_t^{g,N},\iota\mu_t^{g})^q\right]\right\}^{1/q},
\end{align}
where 
$d_{BL}$ is Fortet-Mourier distance (see, e.g. Section 11.2 of \cite{Dudley2002}, page 390), which is in fact defined as:
\begin{align}\label{eq:FM-metric}
d_{BL}(\iota\mu_t^{g,N},\iota\mu_t^{g})=\sup_{\psi\in{\cal R}_1}\left(\int_{D_{\star}}\psi(x)(\iota\mu_t^{g,N}-\iota\mu_t^{g})(dx)\right).
\end{align}
Here, ${\cal R}_1$ is the set of (bounded) Lipschitzian continuous functions $\psi:D_{\star}\to\R$ satisfying $\|\psi\|_{\infty}+\|\psi\|_{\rm Lip}\leq1$ (where $\|\psi\|_{\infty}:=\sup_{x\in D_{\star}}|\psi(x)|$ and $\|\psi\|_{\rm Lip}$ denotes the Lipschitzian coefficient of $\psi$). Let $|\cdot|$ be the Euclidean norm. Then, we can define a metric $d_{\star}$ on $D_{\star}$ as in \cite{Mandelkern89}: fix $x_0\in D$ and define $l(x):=\frac{1}{1+|x-x_0|}$ for $x\in D$. For $x_1,x_2\in D$, define $d_{\star}(x_1,x_2):=|x_1-x_2|\wedge(l(x_1)+l(x_2))$, $d_{\star}(x,\star):=l(x)$ for $x\in D$, and $d_{\star}(\star,\star)=0$. Then, the Lipschitzian coefficient of $\psi$ (as a seminorm) is given by
\[
\|\psi\|_{\rm Lip}=\sup_{x_1\neq x_2,x_1,x_2\in D_{\star}}\frac{\left|\psi(x_1)-\psi(x_2)\right|}{d_{\star}(x_1,x_2)}.
\]
This implies that, for any $\psi\in{\cal R}_1$, and $x_1,x_2\in D$,
\begin{align}\label{eq:LipRn}
\left|\psi(x_1)-\psi(x_2)\right|\leq |x_1-x_2|\wedge(l(x_1)+l(x_2))\leq |x_1-x_2|.
\end{align}
In other words, $\psi\in{\cal R}_1$ is also a (bounded) Lipschitzian continuous function on $D$ with the Lipschitzian coefficient being less than one.

\subsection{Propagator of FPK equation}

This section establishes the propagator corresponding to the transformed FPK equation~\eqref{eq:discountFPK}, which is defined as: for $(t,x)\in[0,T]\times D$,
\begin{align}\label{eq:discountpropagator}
P_{t,T}^gf(x):=\Ex\left[f(X_{T}^{t,x})\exp\left(\int_t^Tg(X_s^{t,x})ds\right)\right],\quad f\in C_b(D),
\end{align}
where the process $X^{t,x}=(X_s^{t,x})_{s\in[t,T]}$ is the unique strong solution of SDE~\eqref{eq:SDEXtx}, i.e.,
\begin{align*}
X_s^{t,x} = x + \int_t^s b(X_r^{t,x})dr + \int_t^s\sigma(X_r^{t,x})dW_r,\quad s\in[t,T].
\end{align*}
The following lemma can be proved by verifying the conditions of Theorem 1 and Lemma 2 in \cite{healthSchweizer2000}.
\begin{lemma}\label{lem:solFKPhi}
Let assumptions {\Aa}, {\Ab} and {\Agl}{\rm-(i)} hold. Suppose also that
\begin{description}
\item[{\AD}] there exists a sequence $(D_k)_{k\in\N}$ of bounded domains with $\overline{D}_k\subset D$ such that $\bigcup_{k=1}^{\infty}D_k=D$, each $D_k$ has a $C^2$-boundary; and for each $k\geq1$, $\sigma\sigma^{\top}(x)$ is uniform elliptic on $\R^n$ for $(t,x)\in[0,T)\times D_k$.
\end{description}
Then, the propagator $P_{t,T}^gf$ defined by \eqref{eq:discountpropagator} satisfies that, for $(t,x)\in(0,T)\times D$,
\begin{align}\label{eq:backwardKol}
\partial_tP_{t,T}^gf(x) + ({\cal A}+g)P_{t,T}^gf(x)=0,\quad P_{T,T}^gf(x)=f(x).
\end{align}
The operator ${\cal A}$ is defined by \eqref{eq:operatorsLn}. Moreover, $P_{\cdot,T}^gf\in C^{1,2}((0,T]\times D)\cap C([0,T]\times D)$ and there exists a unique classical solution of the Cauchy problem~\eqref{eq:backwardKol}.
\end{lemma}

The propagator defined by \eqref{eq:discountpropagator} can be used to establish the following relation satisfied by $\mu_T^{g,N}-\mu_T^g$ for any fixed $T>0$:
\begin{lemma}\label{lem:propagatormugdiff}
Let the conditions of Lemma~\ref{lem:solFKPhi} hold. Then, for any fixed $T>0$, it holds that
\begin{align}\label{eq:propagationPtT}
\big\lc\mu_T^{g,N}-\mu_T^g,f\big\rc&= \big\lc\mu_0^{g,N}-\mu_0^g,P_{0,T}^gf\big\rc\nonumber\\
&\quad+\frac{1}{N}\sum_{i=1}^N\int_0^T\exp\left(\int_0^sg(X_r^i)dr\right)\nabla_xP_{s,T}^gf(X_s^i)^{\top}\sigma(X_s^i)dW_s^i,
\end{align}
where $P_{t,T}^gf$ for $t\in[0,T]$ is the propagator defined by~\eqref{eq:discountpropagator}.
\end{lemma}

\begin{proof}
Recall the state process of the particle system $X^i=(X_t^i)_{t\in[0,T]}$ defined by~\eqref{eq:particlesXi} for $i\geq1$. Lemma~\ref{lem:solFKPhi} allows us to apply It\^o's formula to $P_{t,T}^gf(X_t^i)$ with $t\in[0,T]$, and we have that, for $t\in[0,T]$,
\begin{align*}
P_{t,T}^gf(X_t^i)&=P_{0,T}^gf(X_0^i)+\int_0^t(\partial_s+{\cal A})P_{s,T}^gf(X_s^i)ds+\int_0^t\nabla_xP_{s,T}^gf(X_s^i)^{\top}\sigma(X_s^i)dW_s^i.
\end{align*}
Thus, by Eq.~\eqref{eq:backwardKol} in Lemma~\ref{lem:solFKPhi}, it yields that
\begin{align*}
P_{t,T}^gf(X_t^i)\exp\left(\int_0^tg(X_s^i)ds\right)&=P_{0,T}^gf(X_0^i)\nonumber\\
&\quad+\int_0^t\exp\left(\int_0^sg(X_r^i)dr\right)\nabla_xP_{s,T}^gf(X_s^i)^{\top}\sigma(X_s^i)dW_s^i.
\end{align*}
Recall \eqref{eq:mutgN}. Then, using \eqref{eq:empiricalmeasure}, it follows that
\begin{align*}
\big\lc\mu_t^{g,N},P_{t,T}^gf\big\rc&= \exp\left(\int_0^t\lc\mu_s^N,g\rc ds\right)\big\lc\mu_t^N,P_{t,T}^gf\rc=\frac{1}{N}\sum_{i=1}^NP_{t,T}^gf(X_t^i)\exp\left(\int_0^t g(X_s^{i})ds\right)\nonumber\\ &=\big\lc\mu_0^{g,N},P_{0,T}^gf\big\rc+\frac{1}{N}\sum_{i=1}^N\int_0^t\exp\left(\int_0^sg(X_r^i)dr\right)\nabla_xP_{s,T}^gf(X_s^i)^{\top}\sigma(X_s^i)dW_s^i.
\end{align*}
By equalities \eqref{eq:discountFPK} and \eqref{eq:backwardKol} in Lemma~\ref{lem:solFKPhi}, it holds that
\begin{align}\label{eq:propagatoreqn}
\partial_t\lc\mu_t^g, P_{t,T}^gf\rc=0,\quad {\rm for~all}~t\in[0,T].
\end{align}
Therefore $\lc \mu_t^g,P_{t,T}^gf\rc=\lc \mu_0^g,P_{0,T}^gf\rc$ for all $t\in[0,T]$. This gives that
\begin{align*}
\big\lc\mu_t^{g,N}-\mu_t^g,P_{t,T}^gf\big\rc&= \big\lc\mu_0^{g,N}-\mu_0^g,P_{0,T}^gf\big\rc\nonumber\\
&\quad+\frac{1}{N}\sum_{i=1}^N\int_0^t\exp\left(\int_0^sg(X_r^i)dr\right)\nabla_xP_{s,T}^gf(X_s^i)^{\top}\sigma(X_s^i)dW_s^i.
\end{align*}
In the equality, taking $t=T$ and using $P_{T,T}^gf=f$ given in Lemma~\ref{lem:solFKPhi}, we arrive at \eqref{eq:propagationPtT}. Thus, we complete the proof of the lemma.
\end{proof}

By \eqref{eq:integral-star}, we have that, for all $\psi\in{\cal R}_1$,
\begin{align}\label{eq:integralpsistar}
\int_{D_{\star}}\psi(x)(\iota\mu_t^{g,N}-\iota\mu_t^{g})(dx)=\int_{D}\psi_{\star}(x)(\mu_t^{g,N}-\mu_t^{g})(dx)
=\big\lc\mu_t^{g,N}-\mu_t^{g},\psi_{\star}\big\rc,
\end{align}
where $\psi_{\star}(x):=\psi(x)-\psi(\star)$ for $x\in D$. It follows from \eqref{eq:LipRn} that $\psi_{\star}$ is a (bounded) Lipschitzian continuous function on $D$ with the Lipschitzian coefficient being less than one.
Then, we have the following estimate of the gradient $\nabla_xP_{t,T}\psi_{\star}$ which is given by
\begin{lemma}\label{lem:estimatePtTpsistar}
Let assumptions {\Absig}, {\AD} and {\Agl} hold. Then, there exists a polynomial $\hat{Q}_g=\hat{Q}_{g,b,\sigma,T}:\R_+\to\R_+$ (depending only on $g,b,\sigma,T$) with ${\rm deg}(\hat{Q}_g)={\rm deg}(Q_g)$, such that, for all $\psi\in{\cal R}_1$,
\begin{align}\label{eq:gradientPtTpsistar}
|\nabla_xP_{t,T}\psi_{\star}(x)|\leq \hat{Q}_g(|x|),\quad (t,x)\in[0,T]\times D.
\end{align}
\end{lemma}

\begin{proof}
Note that $\|\psi_{\star}\|_{\infty}\leq 2$ for all $\psi\in{\cal R}_1$. Then, in view of \eqref{eq:discountpropagator}, for all $x_1,x_2\in D$,
\begin{align}\label{eq:P12diff}
&\left|P_{t,T}^g\psi_{\star}(x_1)-P_{t,T}^g\psi_{\star}(x_2)\right|^2\leq 2\Ex\left[\left|\psi_{\star}(X_T^{t,x_1})-\psi_{\star}(X_T^{t,x_2})\right|^2\exp\left(2\int_t^Tg(X_s^{t,x})ds\right)\right]\nonumber\\
&\qquad\qquad+2\Ex\left[\left|\psi_{\star}(X_T^{t,x_2})\right|^2\left|\exp\left(\int_t^Tg(X_s^{t,x_1})ds\right)
-\exp\left(\int_t^Tg(X_s^{t,x_2})ds\right)\right|^2\right]\nonumber\\
&\qquad\quad\leq2\Ex\left[\left|X_T^{t,x_1}-X_T^{t,x_2}\right|^2\right]+8\Ex\left[\left|\int_t^Tg(X_s^{t,x_1})ds-\int_t^Tg(X_s^{t,x_2})ds\right|^2\right]\notag\\
&\qquad\quad\leq2\Ex\left[\left|X_T^{t,x_1}-X_T^{t,x_2}\right|^2\right]+8T\int_t^T\Ex\left[\left|g(X_s^{t,x_1})-g(X_s^{t,x_2})\right|^2\right]ds.
\end{align}
We first note that, it follows from the assumption {\Agl}-(ii) that, for some polynomial $Q_g:\R_+\to\R_+$ and for all $s\in[t,T]$,
\begin{align}\label{eq:gdiffXtx}
\Ex\left[\left|g(X_s^{t,x_1})-g(X_s^{t,x_2})\right|^2\right]&\leq\Ex\left[Q^2_g\left(\left|X_s^{t,x_1}\right|
+\left|X_s^{t,x_2}\right|\right)\left|X_s^{t,x_1}-X_s^{t,x_2}\right|^2\right]\\
  &\leq\left\{\Ex\left[Q^4_g\left(\left|X_s^{t,x_1}\right|+\left|X_s^{t,x_2}\right|\right)\right]\right\}^{\frac12}
  \left\{\Ex\left[\left|X_s^{t,x_1}-X_s^{t,x_2}\right|^4\right]\right\}^{\frac12}.\notag
\end{align}
On the other hand, by the assumption {\Absig}, we have that, for $k\geq1$,
\begin{align}\label{eq:estidifflip}
\Ex\left[\sup_{s\in[t,T]}\left|X_s^{t,x_1}-X_s^{t,x_2}\right|^{2k}\right]\leq |x_1-x_2|^{2k}+C_k\int_t^T \Ex\left[\left|X_s^{t,x_1}-X_s^{t,x_2}\right|^{2k}\right]ds,
\end{align}
where $C_k=C_{b,\sigma,T,k}$ is a positive constant which depends on $b,\sigma,T,k$. Then, the Gronwall's lemma yields that
\begin{align}\label{eq:momentXchaos}
\Ex\left[\sup_{s\in[t,T]}\left|X_s^{t,x_1}-X_s^{t,x_2}\right|^{2k}\right]\leq|x_1-x_2|^{2k}e^{(T-t)C_k}.
\end{align}
Then, it follows from \eqref{eq:P12diff}, \eqref{eq:gdiffXtx} and \eqref{eq:momentXchaos} that, for all $(t,x_1,x_2)\in[0,T]\times D^2$,
\begin{align*}
  &\left|P_{t,T}^g\psi_{\star}(x_1)-P_{t,T}^g\psi_{\star}(x_2)\right|^2\\
  &\quad\leq2\Ex\left[\left|X_T^{t,x_1}-X_T^{t,x_2}\right|^2\right]+8T\int_t^T\Ex\left[\left|g(X_s^{t,x_1})-g(X_s^{t,x_2})\right|^2\right]ds\notag\\
  &\quad\leq2|x_1-x_2|^2e^{(T-t)C_{1}}+8T^2\int_t^T\left\{\Ex\left[Q^4_g(\left|X_s^{t,x_1}\right|+\left|X_s^{t,x_2}\right|)\right]\right\}^{\frac12}
  \left\{\Ex\left[\left|X_s^{t,x_1}-X_s^{t,x_2}\right|^4\right]\right\}^{\frac12}ds\notag\\
  &\quad\leq2e^{(T-t)C_{1}}|x_1-x_2|^2+8T^2e^{\frac{(T-t)C_{2}}
  {2}}|x_1-x_2|^2\int_t^T\left\{\Ex\left[Q^4_g(\left|X_s^{t,x_1}\right|+\left|X_s^{t,x_2}\right|)\right]\right\}^{\frac12}ds.
\end{align*}
It is not difficult to show from \eqref{eq:momentqq} that there exists a polynomial ${\tilde Q}_{g}:={\tilde Q}_{g,b,\sigma,T}:\R_+\to\R_+$ (which depends on $g,b,\sigma,T$ only) such that ${\rm deg}({\tilde Q}_{g})=2{\rm deg}(Q_{g})$, and for all $(t,x_1,x_2)\in[0,T]\times D^2$,
\begin{align*}
  2e^{(T-t)C_{1}}+8T^2e^{\frac{(T-t)C_{2}}{2}}\int_t^T\left\{\Ex\left[Q^4_g(\left|X_s^{t,x_1}\right|+\left|X_s^{t,x_2}\right|)\right]\right\}^{\frac12}ds
  \leq{\tilde Q}_{g}(|x_1|+|x_2|).
\end{align*}
Therefore, we arrive at, for all $(t,x_1,x_2)\in[0,T]\times D^2$,
\begin{align}\label{eq:Ptdifference}
  \left|P_{t,T}^g\psi_{\star}(x_1)-P_{t,T}^g\psi_{\star}(x_2)\right|^2\leq{\tilde Q}_{g}(|x_1|+|x_2|)|x_1-x_2|^2.
\end{align}
Thanks to Lemma~\ref{lem:solFKPhi}, we have that $P_{\cdot,T}^g\psi_{\star}\in C^{1,2}((0,T]\times D)\cap C([0,T]\times D)$ by the fact of $\psi_{\star}\in C_b(D)$. By letting $x_2\rightarrow x_1$ in \eqref{eq:Ptdifference}, it follows that, for all $(t,x)\in[0,T]\times D$,
\begin{align}\label{eq:absoluatevaluegradientchao}
  |\nabla_xP_{t,T}\psi_{\star}(x)|\leq{\tilde Q}^{\frac12}_{g}(2|x|).
\end{align}
Then, the estimate \eqref{eq:gradientPtTpsistar} follows from \eqref{eq:absoluatevaluegradientchao} and the fact that ${\tilde Q}^{\frac12}_{g}(2|x|)\leq\hat{Q}_{g}(|x|)$ for some polynomial $\hat{Q}_g:=\hat{Q}_{g,b,\sigma,T}:\R_+\to\R_+$ with ${\rm deg}({\hat Q}_{g})=\frac12 {\rm deg}({\tilde Q}_{g})={\rm deg}(Q_{g})$. Thus, we complete the proof of the lemma.
\end{proof}

We next discuss the estimate \eqref{eq:gradientPtTpsistar} in Lemma~\ref{lem:estimatePtTpsistar} for the one-dimensional case (i.e., $m=n=1$), but the volatility function $\sigma:D\to\R$ of SDE~\eqref{eq:SDEXtx} is only H\"{o}lder continuous. To this purpose, we impose the following assumption:
\begin{itemize}
  \item[{\Acir}] for $m=n=1$, $b:D\to\R$ is Lipschitiz continuous and $\sigma:D\to\R$ is H\"{o}lder continuous with exponent $\gamma\in[\frac{1}{2},1)$.
\end{itemize}
\begin{lemma}\label{lem:estimatePtTpsistarCIR}
Let assumptions {\Acir}, {\AD} and {\Agl} with ${\rm deg}(Q_g)=0$ hold. Then, there exists a constant $K=K_{g,b,\sigma,T}>0$ such that, for all $\psi\in{\cal R}_1$,
\begin{align}\label{eq:gradientPtTpsistarCIR}
|\nabla_xP_{t,T}\psi_{\star}(x)|\leq K,\quad (t,x)\in[0,T]\times D.
\end{align}
\end{lemma}

\begin{proof}
Note that, under the assumption {\Acir}, the estimate \eqref{eq:estidifflip} in the proof of Lemma~\ref{lem:estimatePtTpsistar} does not hold. However, the assumption {\Acir} implies that SDE~\eqref{eq:SDEXtx} satisfies the Yamada-Watanabe condition given in Proposition 5.2.13 of \cite{karatzasshreve1991}. This yields from this proposition that, for all $s\in[t,T]$, there exists a constant $C=C_{b,\sigma,T}>0$ such that
\begin{align*}
\Ex\left[\left|X_s^{t,x_1}-X_s^{t,x_2}\right|\right]\leq |x_1-x_2|+C\int_t^T \Ex\left[\left|X_s^{t,x_1}-X_s^{t,x_2}\right|\right]ds.
\end{align*}
Then, the Gronwall's lemma yields that, for all $s\in[t,T]$,
\begin{align}\label{eq:momentXchaosCIR}
\Ex\left[\left|X_s^{t,x_1}-X_s^{t,x_2}\right|\right]\leq|x_1-x_2|e^{(T-t)C}.
\end{align}
Note that, under the assumption {\Agl} with ${\rm deg}(Q_g)=0$, the fitness function $g:D\to\R$ is hence Lipschitiz continuous. Similarly to \eqref{eq:P12diff}, we obtain from \eqref{eq:momentXchaosCIR} that
\begin{align*}
\left|P_{t,T}^g\psi_{\star}(x_1)-P_{t,T}^g\psi_{\star}(x_2)\right|&\leq\Ex\left[\left|X_T^{t,x_1}-X_T^{t,x_2}\right|\right]
+2\int_t^T\Ex\left[\left|g(X_s^{t,x_1})-g(X_s^{t,x_2})\right|\right]ds\nonumber\\
&\leq \left(1+2T\|g\|_{\rm Lip}\right)e^{TC}|x_1-x_2|.
\end{align*}
This proves the estimate \eqref{eq:gradientPtTpsistarCIR}.
\end{proof}

Building upon the auxiliary results established by the above lemmas, we next prove the main result of this section on the propagation of chaos of the FPK equation \eqref{eq:FPKeqn} with respect to the metric defined by \eqref{eq:dp}:
\begin{theorem}\label{thm:propagationchaos}
Let assumptions {\Ax}, {\Absig}, {\AD} and {\Agl} {\rm[}or assumptions {\Ax}, {\Acir}, {\AD} and {\Agl} with ${\rm deg}(Q_g)=0${\rm]} hold. Let $\mu=(\mu_{t})_{t\in[0,T]}$ be an arbitrary ${\cal P}(D)$-valued solution of the FPK equation \eqref{eq:FPKeqn} and $\mu^N=(\mu_t^N)_{t\in[0,T]}$ be the ${\cal P}(D)$-valued process given by \eqref{eq:solutionmutN} in Section~\ref{sec:particlesys}. Then, for any $T>0$ and $N\geq1$, there exists a constant $C>0$ which is independent of $N$ such that, for any $p\geq1$,
\begin{align}\label{eq:convergenceratepropagation}
d_{q,T}(\mu,\mu^N)\leq C\left(\alpha(p,q,n,N)+\frac{1}{N^{q-1}}\right),\quad q\geq2,
\end{align}
where the metric $d_{q,T}(\cdot,\cdot)$ is defined by \eqref{eq:dp}, and the first rate of convergence rate $\alpha(p,q,n,N)$ in \eqref{eq:convergenceratepropagation} is given by
\begin{align}\label{eq:converratemu0N}
\alpha(p,q,n,N):=\left\{
                    \begin{array}{ll}
    \displaystyle N^{-\frac{1}{2}}+N^{-\frac{p-q}{p}}, & q>\frac{n}{2},~p\neq2q;\\ \\
    \displaystyle N^{-\frac{1}{2}}\ln(1+N)+N^{-\frac{p-q}{p}}, & q=\frac{n}{2},~p\neq 2q;\\ \\
    \displaystyle N^{-\frac{q}{n}}+N^{-\frac{p-q}{p}}, & q<\frac{n}{2},~p\neq\frac{n}{n-q}.
                    \end{array}
                 \right.
\end{align}
\end{theorem}

\begin{proof}
Using \eqref{eq:propagationPtT} in Lemma~\ref{lem:propagatormugdiff}, it results in, for all $\psi\in{\cal R}_1$,
\begin{align}\label{eq:propagationPtTpsistar}
\big\lc\mu_T^{g,N}-\mu_T^g,\psi_{\star}\big\rc&= \big\lc\mu_0^{g,N}-\mu_0^g,P_{0,T}^g\psi_{\star}\big\rc\nonumber\\
&\quad+\frac{1}{N}\sum_{i=1}^N\int_0^T\exp\left(\int_0^sg(X_r^i)dr\right)\nabla_xP_{s,T}^g\psi_{\star}(X_s^i)^{\top}\sigma(X_s^i)dW_s^i.
\end{align}
We next consider the estimate of the first term of the r.h.s. of the equality~\eqref{eq:propagationPtTpsistar}. Note that
\begin{align}\label{eq:mu0gNmu0}
\mu_0^{g,N}=\mu_0^N=\frac{1}{N}\sum_{i=1}^N\delta_{X_0^i},\quad \mu_0^g=\rho_0.
\end{align}
Let $\phi_0:=P_{0,T}\psi_{\star}$ with $\psi\in{\cal R}_1$. Then, it follows from the assumption {\Ax} that $(\phi_0(X_0^i))_{i\geq1}$ are i.i.d. r.v.s. Let $\tilde{\rho}_0:={\cal L}(\phi_0(X_0^1))$, i.e., the law of the r.v. $\phi_0(X_0^1)$, and define $\tilde\mu^N:=\frac{1}{N}\sum_{i=1}^N\delta_{\phi_0(X_0^i)}$. Then, for $I(x)=x$ for $x\in D$, we obtain from \eqref{eq:converratemu0N0} and the Kantorovich-Rubinstein dual formula (see \cite{Villani2003}) that, for $q\geq2$,
\begin{align}\label{eq:initialdifference0}
\Ex\left[\left|\big\lc\mu_0^{g,N}-\mu_0^g,P_{0,T}^g\psi_{\star}\big\rc\right|^q\right]
&=\Ex\left[\left|\big\lc\mu_0^{g,N}-\mu_0^g,\phi_0\big\rc\right|^q\right]=\Ex\left[\left|\big\lc\tilde\mu^N-\tilde\rho_0,I\big\rc\right|^q\right]\nonumber\\
&\leq\Ex\left[{\cal W}_1(\tilde\mu^N,\tilde\rho_0)^q\right]\leq\Ex\left[{\cal W}_q(\tilde\mu^{N},\tilde\rho_0)^q\right].
\end{align}
By \eqref{eq:discountpropagator} and the fact that $\|\psi_{\star}\|_{\infty}\leq2$, we have from the assumption {\Agl}-(i) (without loss of generality, we have assumed that $g\leq0$) that $\|\phi_0\|_{\infty}\leq2$. Therefore $\int_{D}|x|^p\tilde\rho_0(dx)=\int_{D}|\phi_0(x)|^p\rho_0(dx)\leq 2^p$ for any $p\geq1$. Note that, under the assumption {\Ax}, by Glivenko-Cantelli’s theorem, the empirical measure $\tilde{\mu}^N$ tends weakly to $\tilde{\rho}_0$ as $N\to\infty$. Moreover,  Theorem 1 of \cite{FournierGuillin15} yields that, there is a constant $C$ depending only on $n,p,q$ such that
\begin{align}\label{eq:converratemu0N0}
\Ex\left[{\cal W}_q(\tilde{\mu}_0^N,\tilde{\rho}_0)^q\right]&\leq C\left(\int_{D}|x|^p\tilde{\rho}_0(dx)\right)^{\frac{q}{p}}\alpha(p,q,n,N)
\leq 2^qC\alpha(p,q,n,N),
\end{align}
where ${\cal W}_q$ denotes the Wasserstein metric with order $q\geq2$ on the $q$-th order Wasserstein space ${\cal P}_q(D)$ (see \cite{Villani2003}), and the rate of convergence of rate is given by \eqref{eq:converratemu0N}. Then, in view of \eqref{eq:initialdifference0} and \eqref{eq:converratemu0N0}, there is a constant $C$ depending only on $n,p,q$ such that
\begin{align}\label{eq:initialdifference}
\Ex\left[\left|\big\lc\mu_0^{g,N}-\mu_0^g,P_{0,T}^g\psi_{\star}\big\rc\right|^q\right]\leq C\alpha(p,q,n,N).
\end{align}
On the other hand, apply the estimate \eqref{eq:gradientPtTpsistar} in Lemma~\ref{lem:estimatePtTpsistar} or the estimate \eqref{eq:gradientPtTpsistarCIR} in Lemma~\ref{lem:estimatePtTpsistarCIR}. Then, for the estimate of the second term of the r.h.s. of the equality~\eqref{eq:propagationPtTpsistar}, under the assumption {\Ax}, the BDG inequality yields that, for $q\geq2$,
\begin{align}\label{eq:estimate2}
&\Ex\left[\sup_{t\in[0,T]}\left|\frac{1}{N}\sum_{i=1}^N\int_0^t\exp\left(\int_0^sg(X_r^i)dr\right)
\nabla_xP_{s,T}^g\psi_{\star}(X_s^i)^{\top}\sigma(X_s^i)dW_s^i\right|^q\right]\nonumber\\
&\qquad \leq \frac{C}{N^q}\sum_{i=1}^N\Ex\left[\left(\int_0^T\exp\left(2\int_0^sg(X_r^i)dr\right)
\left|\nabla_xP_{s,T}^g\psi_{\star}(X_s^i)\right|^2\left|\sigma(X_s^i)\right|^2ds\right)^{q/2}\right]\nonumber\\
&\qquad\leq\frac{C}{N^q}\sum_{i=1}^N\Ex\left[\left(\int_0^T\exp\left(2\int_0^sg(X_r^i)dr\right)
\hat{Q}^2_{g}(|X_s^i|)\left|\sigma(X_s^i)\right|^2ds\right)^{q/2}\right]\nonumber\\
&\qquad\leq\frac{C}{N^q}\sum_{i=1}^N\Ex\left[\left(\int_0^T
\hat{Q}^2_{g}(|X_s^i|)\left|\sigma(X_s^i)\right|^2ds\right)^{q/2}\right]\notag\\
&\qquad=\frac{C}{N^{q-1}}\Ex\left[\left(\int_0^T
\hat{Q}^2_{g}(|X_s^1|)\left|\sigma(X_s^1)\right|^2ds\right)^{q/2}\right],
\end{align}
for some constant $C=C_{g,b,\sigma,q,T}>0$ which may be different from line to line, where we used $g\leq0$ in the assumption {\Agl}, and $\hat{Q}_g$ is the polynomial with ${\rm deg}(\hat{Q}_g)={\rm deg}(Q_g)$ given in Lemma~\ref{lem:estimatePtTpsistar}. In addition, by applying the linear growth condition in the assumption~{\Absig} or the H\"older continuity given in the assumption~{\Acir}, there exists a constant $C=C_{g,b,\sigma,T}>0$ such that, $\Px$-a.s.
\begin{align}\label{eq:estimatepoly1}
  \hat{Q}^2_{g}(|X_s^1|)\left|\sigma(X_s^1)\right|^2\leq {\rm P}_1\left(\sup_{t\in[0,T]}|X^1_t|^2\right),\quad s\in[0,T],
\end{align}
where ${\rm P}_1:\R_+\to\R_+$ is a polynomial with ${\rm deg}({\rm P}_1)={\rm deg}(Q_g)+1$. By \eqref{eq:momentqq} under the assumption {\Ax}, it follows that
\begin{align}\label{eq:momentqq11}
  \Ex\left[\sup_{t\in[0,T]}\left|X^1_t\right|^{{\rm deg}({\rm P}_1)q}\right]\leq C\left\{\Ex\left[\left|X^1_0\right|^{{\rm deg}({\rm P}_1)q}\right]+1\right\}<\infty,
\end{align}
for some constant $C=C_{g,b,\sigma,T}>0$. Then, by applying \eqref{eq:estimate2}, \eqref{eq:estimatepoly1} and \eqref{eq:momentqq11}, there exists a constant $C=C_{g,b,\sigma,T,q}>0$ which may be different from line to line such that
\begin{align}\label{eq:estimate200}
&\Ex\left[\sup_{t\in[0,T]}\left|\frac{1}{N}\sum_{i=1}^N\int_0^t\exp\left(\int_0^sg(X_r^i)dr\right)
\nabla_xP_{s,T}^g\psi_{\star}(X_s^i)^{\top}\sigma(X_s^i)dW_s^i\right|^q\right]\nonumber\\
&\qquad\qquad\leq\frac{C}{N^{q-1}}\Ex\left[\left(\int_0^T
\hat{Q}^2_{g}(|X_s^1|)\left|\sigma(X_s^1)\right|^2ds\right)^{q/2}\right]\nonumber\\
&\qquad\qquad
\leq\frac{C}{N^{q-1}}\left\{{\mathbb E}\left[\left|X^1_0\right|^{{\rm deg}({\rm P}_1)q}\right]+1\right\}.
\end{align}
Then, the estimate \eqref{eq:convergenceratepropagation} can be deduced from \eqref{eq:estimate200} and \eqref{eq:initialdifference} jointly. Thus, we complete the proof of the theorem.
\end{proof}

Recall that Example~\ref{exam:one-dimCIR} relates to the one-dimensional RM eqaution with $b(x)=a+bx$, $\sigma(x)=\sigma\sqrt{x}$ and $g(x)=-x$ for $x\in D=(0,\infty)$; Example~\ref{exam:polynomialcff} relates to the one-dimensional RM eqaution with $b(x)\equiv0$, $\sigma(x)=\sigma\sqrt{2}$ and $g(x)=-x^{2q} + \sum_{l=0}^{2q-1}\alpha_lx^l$ given by \eqref{eq:polynomialconfitfcn} for $x\in D=(-\infty,\infty)$; Example~\ref{exam:affineprocess} relates to the $n$-dimensional RM equation with $b(x)=b+ Bx$, $\sigma(x)\equiv\sigma$, and $g(x)=-r(x)$ given by \eqref{eq:fitness-g-2} for $x\in D=\R^n$. It is not difficult to verify the validity of assumptions {\Acir}, {\AD} and {\Agl} with ${\rm deg}(Q_g)=0$ for Example~\ref{exam:one-dimCIR}, and  the validity of assumptions {\Absig}, {\AD} and {\Agl} for Example~\ref{exam:polynomialcff} and Example~\ref{exam:affineprocess}. Therefore, the propagation of chaos with the rate of convergence \eqref{eq:convergenceratepropagation} in Theorem \ref{thm:propagationchaos} under the assumption {\Ax} on the initial data of the particle system holds for these examples.


\begin{thebibliography}{}
		
\bibitem[Alfaro and Carles(2014)]{alfarocarles14} Alfaro, M., and R. Carles (2014): Explicit solutions for replicator-mutator equations: extinction versus acceleration. {\it SIAM J. Appl. Math.} 74, 1919-1934.

\bibitem[Alfaro and Carles(2017)]{alfarocarles17} Alfaro, M., and R. Carles (2017): Replicator-mutator equations with quadratic fitness. {\it Proc. Amer. Math. Soc.} 145, 5315-5327.
		
\bibitem[Alfaro and Veruete(2019)]{alfaroveruete19} Alfaro, M., and M. Veruete (2019): Evolutionary branching via replicator-mutator equations. {\it J. Dyn. Diff. Eqn.} 31, 2029-2052.

\bibitem[B\"urger(1988)]{Burger88} B\"urger, R. (1988): Perturbations of positive semigroups and applications to population genetics. {\it Math. Z.} 197, 259-272.

\bibitem[B\"urger(1998)]{Burger98} B\"urger, R. (1998): Mathematical properties of mutation-selection models. {\it Genetica} 102, Article number: 279.

\bibitem[Del Moral and Miclo(2000)]{DelMoralMiclo2000} Del Moral, P., and L. Miclo (2000): A Moran particle system approximation of Feynman-Kac formulae. {\it Stoch. Process. Appl.} 86, 193-216.

\bibitem[Dudley(2002)]{Dudley2002} Dudley, R.M. (2002): {\it Real Analysis and Probability}. Cambridge Studies in Advanced Mathematics Book 74, Cambridge University Press, Cambridge.

\bibitem[Duffie, et al.(2003)]{duffieaffineaap} Duffie, D.,  D. Filipovi\'c, and W. Schachermayer (2003): Affine processes and applications in finance. {\it Ann. Appl. Probab.} 13, 984-1053.

\bibitem[Fleming(1979)]{Fleming79} Fleming, W.H. (1979): Equilibrium distributions of continuous polygenic traits. {\it SIAM J. Appl. Math.} 36, 148-168.


\bibitem[Fournier and Guillin(2015)]{FournierGuillin15} Fournier, N., and A. Guillin (2015): On the rate of convergence in Wasserstein distance of the empirical measure. {\it Probab. Theory \& Rel. Fields} 162, 707-738.

\bibitem[Health and Schweizer(2000)]{healthSchweizer2000} Health, D., and M. Schweizer (2000): Martingale versus PDEs in finance: An euivalence result with examples. {\it J. Appl. Probab.} 37, 947-957.

\bibitem[Kac(1956)]{Kac1956} Kac, M. (1956): {\it Foundations of Kinetic Theory}. In Proceedings of the Third Berkeley Symposium on Math. Stats. \& Probab. 1954-1955, vol. III, 171-197. University of California Press, Berkeley and Los Angeles.

\bibitem[Karatzas and Shreve(1991)]{karatzasshreve1991} Karatzas, I., and S.E. Shreve (1991): {\it Borwonian motion and Stochastic Calculus}. Springer-Verlag, New York.

\bibitem[Kimura(1965)]{Kimura65} Kimura, M. (1965): A stochastic model concerning the maintenance of genetic variability in quantitative characters. {\it Proc. Nat. Acad. Sci.} 54, 731-736.

\bibitem[Mandelkern(1989)]{Mandelkern89} Mandelkern, M. (1989): Metrization of the one-point compactification. {\it Proceedings of AMS} 107, 1111-1115.

\bibitem[Pinsky(1995)]{pinsky95} Pinsky, R.G. (1995): {\it Positive Harmonic Function and Diffusion}. Cambridge University Press, Cambridge.

\bibitem[Qin and Linetsky(2016)]{qinLinetsky16} Qin, L., and V. Linetsky (2016): Positive eigenfunctions of Markovian pricing operators: Hansen-Scheinkman factorization and Ross recovery. {\it Opers. Res.} 64, 99-117.

\bibitem[Rouzine et al.(2008)]{RouzineBrunetWilke08} Rouzine, I., E. Brunet, and C. Wilke (2008): Thetraveling-wave approach to asexual evolution: Muller's ratchet and speed of adaptation. {\it Theor. Popul. Biol.} 73, 24-46.

\bibitem[Rouzine et al.(2003)]{RouzineWakekeyCoffin03} Rouzine, I., J. Wakekey, and J. Coffin (2003): The solitary wave of asexual evolution. {\it Proc. Natl. Acad. Sci.} 100, 587-592.

\bibitem[Sniegowski and Gerrish(2010)]{SniegowskiGerrish10} Sniegowski, P., and P. Gerrish (2010): Beneficial mutations and the dynamics of adaptation in asexual populations. {\it Phil. Trans. Royal Soc. B} 365, 1255-1263.

\bibitem[Sznitman(1991)]{Sznitman91} Sznitman, A.S. (1991): {\it Topics in Propagation of Chaos}. In \`Ecole D'Et\'e de Probabilit\'es de Saint-Flour XIX-1989. Lecture Notes in Math. 1464, 165-251. Springer, Berlin.

\bibitem[Tsimring(1996)]{Tsimringetal96} Tsimring, L., H. Levine, and D. Kessler (1996): RNA virus evolution via a fitness-space model. {\it Phys. Rev. Lett.} 76, 4440-4443.

\bibitem[{Villani(2003)}]{Villani2003} Villani, C. (2003): {\it Topics in Optimal Transportation.} Graduate Studies in Mathematics, Volume: 58, AMS.

\bibitem[Xu(2018)]{Xu2018} Xu, L.P. (2018): Uniqueness and propagation of chaos for the Boltzmann equation with moderately soft potentials. {\it Ann. Appl. Probab.} 28, 1136-1189.

\end{thebibliography}
\end{document}